\documentclass[12pt]{amsart}
\usepackage{color}
\usepackage{amssymb}
\usepackage{comment}
\usepackage{pdfpages}
\usepackage{graphicx}
\usepackage{dcolumn}

\usepackage{graphicx}
\usepackage{tikz} 
\usepackage{dcolumn}
\usetikzlibrary{arrows,positioning}
\tikzset{
    >=stealth',
    pil/.style={
           ->,
           thick,
           shorten <=2pt,
           shorten >=2pt,}
}

\usepackage{amsthm}
\usepackage{mathrsfs}
\usepackage[all]{xy}
\usepackage{graphicx}
\usepackage{cite}
\usepackage{mhequ}
\usepackage{algorithm}
\usepackage{algorithmic}
\RequirePackage[numbers]{natbib}
\RequirePackage[colorlinks=true, pdfstartview=FitV, linkcolor=blue,
  citecolor=blue, urlcolor=blue]{hyperref}
\RequirePackage{hypernat}
\usepackage{paralist}

\usepackage{graphicx}
\graphicspath{{./figures/}}
\usepackage{amsfonts}
\usepackage{amsmath}
\usepackage{amsthm}
\usepackage{amssymb}
\usepackage{amsbsy}
\usepackage{epsfig}
\usepackage{fullpage}
\usepackage{natbib, mathrsfs} 
\usepackage{verbatim}
\usepackage[latin1]{inputenc}
\usepackage{mhequ}
\usepackage{algorithm}
\usepackage{algorithmic}

\numberwithin{equation}{section}

\newtheorem{theorem}{Theorem}[section]
\newtheorem{lemma}[theorem]{Lemma}
\newtheorem{remark}[theorem]{Remark}

\newtheorem{prop}[theorem]{Proposition}

\newtheorem{assumption}[theorem]{Assumption}

\theoremstyle{plain}
\newtheorem{thm}{Theorem}
\newtheorem*{thm-non}{Theorem}
\newtheorem{cor}[thm]{Corollary}

\newtheorem{example}[thm]{Example}

\usepackage{tikz} 
\usepackage{dcolumn}
\usetikzlibrary{arrows,positioning}
\tikzset{
    >=stealth',
    pil/.style={
           ->,
           thick,
           shorten <=2pt,
           shorten >=2pt,}
}

\usepackage{amsthm}
\usepackage{mathrsfs}
\usepackage[all]{xy}
\usepackage{algorithm}
\usepackage{algorithmic}

\def \be{\begin{equs}}
\def \ee{\end{equs}}
\def \P{\mathbb{P}}
\def \E{\mathbb{E}}

\def \lip{\mathrm{Lip}}
\def \MH{\mathrm{MH}}
\def \EE{\mathrm{EE}}
\def \Om{\mathbf{\Omega}}

\def \Do{D_\Omega}
\def \he{\mathcal{H}_\epsilon(x)}

\theoremstyle{definition}

\begin{document}

\title[Curvature Adaptive MC]{Finite Sample Properties of Adaptive Markov Chains via Curvature}

\author{Natesh S. Pillai$^{\ddag}$}
\thanks{$^{\ddag}$pillai@fas.harvard.edu, 
   Department of Statistics
    Harvard University, 1 Oxford Street, Cambridge
    MA 02138, USA}

\author{Aaron Smith$^{\sharp}$}
\thanks{$^{\sharp}$smith.aaron.matthew@gmail.com, 
   Department of Mathematics and Statistics
University of Ottawa, 585 King Edward Drive, Ottawa
ON K1N 7N5, Canada}
 
\begin{abstract}
Adaptive Markov chains are an important class of Monte Carlo methods for sampling from probability distributions. The time evolution of adaptive algorithms depends on past samples, and thus these algorithms are non-Markovian. Although there has been previous work establishing conditions for their ergodicity, not much is known theoretically about their finite sample properties. In this paper, using a notion of discrete Ricci curvature for Markov kernels introduced by Ollivier, we establish concentration inequalities and finite sample bounds for a class of adaptive Markov chains. After establishing some general results, we give quantitative bounds for `multi-level' adaptive algorithms such as the equi-energy sampler. We also provide the first rigorous proofs that the finite sample properties of an equi-energy sampler are superior to those of related parallel tempering and Metropolis-Hastings samplers after a learning period comparable to their mixing times. 
\end{abstract}
 \maketitle

\section{Introduction} 
Markov Chain Monte Carlo (MCMC) algorithms are indispensable for sampling from complex probability distributions. Often the implementation of these algorithms requires the tuning of a few parameters for optimal performance (see \cite{GeRoGi96} for a well-known guide to tuning simple MCMC algorithms). In this context, adaptive MCMC algorithms were originally developed \cite{haar:99,haar:01} so as to automate the tuning process; see \cite{atch:11} for a useful review.    
In recent years, many interesting adaptive algorithms \cite{robe:rose:09,KZW06} have been developed with a view to improving the efficiency of existing algorithms. For instance, the equi-energy sampler \cite{KZW06} builds on the parallel tempering algorithm \cite{Geye91} and constructs proposals based on `equi-energy' moves which lead to efficient jumping between the modes for multi-modal distributions. \par
Adaptive algorithms usually entail proposal distributions which are constructed from past samples. These algorithms are no longer Markovian. Thus, many of the sophisticated tools developed to study the convergence of MCMC algorithms (see, \textit{e.g.}, \cite{AlFi94, LPW09} for surveys) do not apply directly for adaptive algorithms. In particular, due to the discreteness of empirical measures, for many adaptive algorithms it can be shown that their corresponding kernels do not converge to any limiting kernel in the Total Variation metric. This suggests using other metrics to study mixing of adaptive Markov chains. A natural candidate is the slightly weaker Wasserstein metric \cite{JoOl10,hair08}, as it is is weak enough for convergence results to hold but strong enough to give useful bounds on the properties of finite samples.

\subsection{Previous Work on Ergodicity of Adaptive MCMC} 
It is known that adaptive algorithms do not always converge to the target distribution (see e.g. \cite {robe:rose:07,latu:13} for two examples). The paper \cite{robe:rose:07}  works out sufficient conditions ensuring ergodicity (also see \cite{atchade2005, and2006}). Briefly speaking, \cite{robe:rose:07} requires two conditions for convergence -- containment and diminishing adaptation. The containment condition requires a uniform control on the family of transition kernels indexed by the parameter of adaptation. In a recent paper \cite{Krys13}, the authors show that adaptive algorithms failing to satisfy the containment condition will perform poorly. Diminishing adaptation requires the algorithm to eventually stop adapting. Often in practice, verifying the diminishing adaptation is easier than the containment condition (see \cite{saks2010} and \cite{viho2011a, viho2011b}); but diminishing adaptation is not always needed for convergence. 

 More recent work, such as \cite{fort2011a, fort2011b}, have further developed sufficient conditions for the law of large numbers and central limit theorem to hold for general adaptive algorithms; these conditions often hold for the equi-energy sampler.  \par
Although the above results are quantitative, they do not give any finite sample bounds or information on mixing rates. To our knowledge, the only paper to do so is  \cite{ScWo13}. In \cite{ScWo13}, the authors derive conductance based proofs for showing lower bounds on an analogue of the mixing times of some adaptive Markov chains for multi-modal distributions. Our results are complementary, providing both finite sample bounds and in some cases giving quantitatively comparable upper bounds on an analogue to the mixing time. 
\subsection{Our Contribution}
We first establish concentration inequalities for small (and possibly time-dependant) perturbations of Markov chains that satisfy a fairly weak contraction condition. Despite the fact that adaptive Markov chains are not in fact Markov chains, these bounds can be used to obtain finite sample bounds for Monte Carlo samples obtained from adaptive Markov chains, and thus compare the efficiency of adaptive algorithms to their underlying Markov chains.
Our main tool is based on the notion of  `curvature' for Markov chains introduced in \cite{Olli09} and going back to the work of \cite{BuDy97} and others; this idea of curvature of Markov kernels is strictly more general than the `Doeblin condition' (see \cite{Doeb38}).  \par 

The finite sample bounds obtained in this paper can be obtained for many classes of adaptive Markov chains. In this paper, we focus exclusively on the equi-energy sampler \cite{KZW06}. Our contribution here is twofold. We first construct an equi-energy sampler from a random walk Metropolis algorithm for a simple family of multi-modal target distributions. We show that, after an appropriate burn-in period, this equi-energy sampler has a faster decaying autocorrelation function than that of a related parallel tempering sampler \cite{Geye91} as well as the underlying random walk Metropolis algorithm. Unsurprisingly, we also find that the initial burn-in period required by the equi-energy sampler is substantially smaller than the mixing time of the underlying random walk Metropolis algorithm. However, we find that for our example the burn-in period required by the equi-energy sampler is not substantially smaller than the mixing time of the corresponding parallel tempering algorithm. To summarize, we find for some examples and some measures of efficiency that the equi-energy sampler is substantially better than the underlying MCMC and is \textit{eventually} much better than parallel tempering, but it does not seem to substantially improve on the initial mixing period relative to parallel tempering. To our knowledge, this gives the first rigorous comparison of the mixing properties of an adaptive Markov chain to its Markovian version.  \par 

Next, we analyze the convergence of the equi-energy sampler for a broad class of target distributions. The equi-energy sampler is demonstrated to have good properties in practice \cite{KZW06}, and there is some knowledge of asymptotic variances of estimators constructed from adaptive samplers \cite{Atch10}. However, there has been very little work on bounding the errors of estimators obtained from finite runs of this algorithm (although see \cite{ScWo13} for lower bounds on the mixing time that are complementary to our upper bounds, and \cite{AtWa12} for asymptotic upper bounds). We provide concentration bounds for a broad class of equi-energy samplers, after a suitable burn-in period. \par 
Recall that every equi-energy sampler has an associated `limiting' Markov chain; the equi-energy sampler is designed to mimic this chain. Our main result here shows that kernels associated with the equi-energy sampler converge to the kernel of this limiting Markov chain in the Wasserstein metric. As mentioned above, this convergence \emph{does not} happen in the Total Variation metric. All our mixing and concentration results for equi-energy samplers are inherited from, and inferior to, their associated limiting Markov chains. As suggested by Atchade \cite{Atch10}, this is fundamental to the nature of equi-energy algorithms, and not merely a technical issue. Our approach is less general but quantitatively stronger than that of \cite{KZW06} and the related argument by Atchad{\'e} and Liu in the discussion of \cite{KZW06}. It is slightly different in its application than the later convergence results of \cite{HuKo12}, which is limited to discrete state spaces, and \cite{AJDM08}, which makes minorization assumptions that are stronger than our curvature assumptions. Our proof of ergodicity is based on coupling; for stochastic systems with memory,  coupling arguments similar in spirit to ours, but in a different setting, can be found in \cite{Hair02,Matt02}. 

\subsection{Paper Guide}
The rest of the paper is organized as follows.
In Section \ref{SecNotation}, we set notation and review the notion of curvature for Markov chains. In Section \ref{SecConc}, we prove bounds that allow this notion of curvature to be used in the analysis of adaptive chains. In Section \ref{SecExEe}, we study the equi-energy sampler in detail.
 We close with a brief discussion of open problems that arise from our work. Some technical proofs are gathered in the Appendix.

\section{Preliminaries} \label{SecNotation}
\subsection{Notation}
For a random variable $X$ and measure $\mu$, $X \sim \mu$ denotes that $X$ is distributed according to $\mu$. $\mathrm{Unif}(A)$ denotes uniform distribution on the set $A$. For a function $f$ on a metric space, we will use $\| f \|_{\lip}$ to denote the Lipschitz constant of $f$.
For any distribution $\mu$ and function $f$, we write $\E_{\mu}[f]$ and $\mathrm{Var}_{\mu}[f]$ for the mean and variance of $f$ with respect to $\mu$. We will denote the support of $\mu$ by $\mathrm{Supp}(\mu)$. For any Markov transition kernel $T$, we write $(Tf)(x)$ to mean the usual averaging $\int f(y) T(x,dy)$.

 We will write $f = O(g)$ or $f = \Om(g)$ to mean that there exists a constant $\mathcal{C} > 0$ so that $f(x) \leq \mathcal{C} g(x)$ or $f(x) \geq \mathcal{C} g(x)$ respectively. We also write $f = o(g)$ if $\lim_{x \rightarrow \infty} \frac{f(x)}{g(x)} = 0$. Finally, $f = \Theta(g)$ means $f = O(g)$ and $f = \Om(g)$. 

\subsection{Curvature} \label{SubsecCurvNot}

We use the framework of curvature for operators used heavily in \cite{JoOl10} and introduced in \cite{Olli09}. Throughout, we will consider several kernels $K$ on several Polish spaces $(\Omega, d)$. Associated with each kernel and Polish space is a notion of curvature. Fix two measures $\mu, \nu$ on $\Omega$, and let $\Pi(\mu,\nu)$ be the set of all couplings of $\mu$ and $\nu$. The \textit{Wasserstein distance} between $\mu$ and $\nu$ is defined as
\be
W_{d}(\mu,\nu) = \inf_{\zeta \in \Pi(\mu,\nu)} \int_{x,y \in \Omega} d(x,y) \zeta(dx,dy).
\ee
In this paper, we frequently pass between this definition of the Wasserstein distance and the following version provided by the Kantorovitch-Rubinstein duality theorem (see Remark 6.5 of \cite{Vill08}):
\be 
W_{d}(\mu, \nu) = \sup_{\| f \|_{\lip} = 1} \vert\mu(f) - \nu (f) \vert.
\ee 
The \emph{Ricci curvature} of the kernel $K$ at the pair of points $x,y$ is defined to be:
\be
\kappa(x,y) = 1 - \frac{W_{d}(K(x, \cdot), K(y, \cdot))}{d(x,y)}
\ee
and the curvature of the entire chain is defined to be
\be
\kappa = \inf_{x,y \in \Omega} \kappa(x,y).
\ee

It is worth noting that in many cases of interest, it is sufficient to calculate $\kappa(x,y)$ for $d(x,y)$ `small'; see, \textit{e.g.}, Prop 19 of \cite{Olli09}.  In general, $\kappa $ can take any value in $ [-\infty, 1]$. \\

In addition to the curvature, which describes the tendency of nearby points to coallesce, we also need several measures of variation from \cite{JoOl10}. The \textit{eccentricity} of a point $x \in \Omega$ is given by:
\be
E(x) = \int_{\Omega} d(x,y) \pi(dy).
\ee
The \textit{coarse diffusion} is defined as
\be
\sigma^{2}(x) = \frac{1}{2} \int_{y,z \in \Omega} d(y,z)^{2} K(x,dy) K(x,dz),
\ee
the \textit{local dimension} is given by 
\be
n(x) = \inf_{\| f \|_{\mathrm{Lip}} = 1} \frac{\int_{y,z \in \Omega} d(y,z)^{2} K(x,dy) K(x,dz)}{\int_{y,z \in \Omega} \vert f(y) - f(z) \vert^{2} K(x,dy) K(x,dz)},
\ee
and the \textit{granularity} is
\be
\sigma_{\infty} = \frac{1}{2} \sup_{x \in \Omega} \, \mathrm{diam} \, \mathrm{Supp}(K(x,\cdot)).
\ee
\par For  any  $f: \Omega \mapsto \mathbb{R}$, burn-in time $T_b \geq 0$ and running time $T \geq 1$, define
\be \label{eqn:MCav}
 \hat{\pi}_{T,T_{b}}(f) =  {1 \over T} \sum_{t=T_{b} + 1}^{T_{b} + T} f(X_{t}).
\ee

 Fix a Markov chain $X_{t}$ with associated operator $K$ and invariant measure $\pi$ on the Polish space $(\Omega, d)$, and define the quantity
 \be
 V^2 = {1 \over \kappa T} \left(1+ {T_b \over T} \right) \sup_{x \in \Omega} {\sigma^2(x) \over n(x) \kappa}. 
 \ee 
 Theorem 4 of \cite{JoOl10} gives the following concentration inequality for any Lipschitz function $f$ with $\pi(f) =0$:
\be \label{eqn:Wasscon}
\P\left( \frac{\vert \widehat{\pi}_{T,T_b}(f) - \E_x[\widehat{\pi}_{T,T_b}(f)] \vert}{\| f \|_{\mathrm{Lip}}} > r \right) \leq 2 \,e^{-r^2 / (16 V^2)}
\ee
for $x \in \Omega$ and $r < r_{\max} =  {4V^2 \kappa T\over 3 \sigma_\infty}$. A similar result holds for $r > r_{\max}$. The subscript $x$ in $\E_x$ denotes the initial condition of $X_t$. 
%
\section{Concentration for Time-Inhomogenous Markov Chains via curvature} \label{SecConc}
In this section, our goal is to derive a concentration inequality of the form \eqref{eqn:Wasscon} for time-inhomogeneous Markov chains. The main result in this section, Theorem \ref{ThmConcAlmostPositiveAlt} below, will be our key tool in obtaining finite time bounds for adaptive Markov chains. We consider a time-inhomogenous Markov chain $\{X_{t} \}_{t \in \mathbb{N}}$ being driven by a sequence of transition kernels $\{ K_{t} \}_{t \in \mathbb{N}}$. By this, we mean that for measurable sets $A \subset \Omega$,
\be
\mathbb{P}(X_{t+1} \in A| \{X_s\}_{s \leq t}) = K_t(X_t, A).
\ee
For all $t_{1} \leq t_{2}$ we use the following convention for the order of kernels in a product:
\be \label{EqMatProdDef1}
\prod_{s = t_{1}}^{t_{2}} K_{s} \equiv K_{t_{2}} K_{t_{2}-1} \ldots K_{t_{1}}.
\ee 

In all of our bounds in this section, each kernel $K_{t}$ is assumed to be a small perturbation of a particular transition kernel $K$ with good mixing properties, so that $\{ X_{t} \}_{t \in \mathbb{N}}$ is a small perturbation of a time-homogenous Markov chain being driven by $K$. Although adaptive MCMC samplers are not generally Markov chains at all, we will show in later sections that these bounds can be applied to adaptive Markov chains via a conditioning argument. This notion of perturbation of a kernel with positive curvature is equivalent to that mentioned in Problem O in the survey article \cite{Olli10}. 

\subsection{Concentration inequality for time-inhomogeneous Markov chains}
Let $K$ be a kernel defined on a Polish space $(\Omega, d)$ and with curvature $\kappa > 0$. Let $\{ K_t \}_{t \in \mathbb{N}}$ be a sequence of kernels, and $X_t$ be a time-inhomogeneous Markov chain driven by $K_t$. The granularity of the kernels $\{ K_t \}_{t \in \mathbb{N}}$ are respectively denoted by $\{ \sigma_{t, \infty} \}_{t \in \mathbb{N}}$. Define 
\be
\Sigma_{T_b,T,\infty} = \max_{T_b \leq t \leq T_b + T} \sigma_{t,\infty}.
\ee
Recall that $T_b \geq 0$ is our burn-in time. Let $\mathcal{V}(x)$ be a $\mathcal{C}_{v}$-Lipschitz function satisfying
\be \label{eqn:V}
\mathcal{V}(x)   \geq \frac{1}{\kappa} \mathrm{Var} K_t(x,\cdot), \quad \quad  T_b \leq t \leq T_b + T.
\ee
This function will be used to bound the amount that a single step of the Markov chain can influence the sum in Equation \eqref{eqn:MCav}, and plays a similar role to the bound on absolute differences that appears in Azuma's inequality and other concentration results; in practice it is often easiest to bound $\mathcal{V}$ using the coarse diffusion and granularity of the Markov kernel.

Define the constants\footnote{we have not optimized these constants; these values suffice for our purpose.} 
\be[eqn:lmax]
\delta_{\max} &=  \frac{\kappa}{480 \mathcal{C}_{v} + 240}, \\
\lambda_{\max} &= \kappa T \min \Big( \frac{1}{16 \mathcal{C}_{v}}, \frac{1}{6 \Sigma_{T_b,T,\infty}}, \frac{1}{36} \Big). 
\ee 
The following is the main result of this section. 
\begin{thm}  \label{ThmConcAlmostPositiveAlt}
Let $\mathcal{V}(x)$, $\delta_{\max}, \lambda_{\max}$ be as in Equations \eqref{eqn:V} and \eqref{eqn:lmax} respectively, fix  $T_b, T \in \mathbb{N}$ and $0 < \delta < \delta_{\max}$, and assume that \be \label{IneqOldLemma6WassApprox2}
\sup_{x \in \Omega} W_{d}(K(x,\cdot), K_{t}(x, \cdot)) \leq \delta, \quad T_b < t \leq T_b + T.
\ee 
  Then for all distributions $\mu$ of the chain at time $T_{b}$, all $r>0$, $0 < \lambda \leq \lambda_{\max}$, and all $1$-Lipschitz functions $f$, we have 
\be \label{InequalityConcBoundMain}
\P_{\mu} \Big( \vert \hat{\pi}_{T,T_b}(f) - \E_{X_{T_{b}} \sim \mu}[\hat{\pi}_{T,T_b}(f)] \vert \geq r\Big) \leq 2 e^{-\lambda r} e^{\frac{4 \lambda^{2}}{\kappa T^{2}} \sum_{\ell = 0}^{T-1} \E_{X_{T_{b}} \sim \mu}[\mathcal{V}(X_{T_{b} + \ell})]}.
\ee 

\end{thm}
If there exists a single kernel $K$ so that $K_{t} = K$ for all $t \in \mathbb{N}$, choosing the optimal value of $0 < \lambda < \lambda_{\max}$ almost reduces our result to Theorem 5 of \cite{JoOl10} (on which we model both the statement and much of the proof of our Theorem \ref{ThmConcAlmostPositiveAlt}). For the reader's convenience we point out the key differences:
\begin{itemize}
\item Our value of $\lambda_{\max}$ (equivalent to $r_{\max}$ of \cite{JoOl10}) is slightly smaller; it gives up a factor of around 6 for the two terms that occur in their bound, and also has an absolute bound of $\frac{1}{36}$ that does not appear in Theorem 5 of \cite{JoOl10} (in practice, this absolute bound does not seem to be significant). This loss of power occurs when we bound a sum of several terms by their maximum or vice versa, where the corresponding `sum' or `maximum' in \cite{JoOl10} had only a single term. 
\item In \cite{JoOl10}, our inequality \eqref{eqn:V} was replaced by the inequality 
\be \label{IneqCompToJoOl} 
\frac{\sigma(x)^{2}}{\eta_{x} \kappa} \leq \mathcal{V}(x)
\ee 
where $\sigma, \eta$ are coarse diffusion and the local dimension of the Markov kernel respectively. As noted in \cite{JoOl10}, any function that satisfies inequality \eqref{eqn:V} also satisfies inequality \eqref{IneqCompToJoOl} and can be used in Theorem 5 of \cite{JoOl10}; this is a cosmetic difference.
\end{itemize}

\begin{remark}
Although the different levels of the Equi-Energy sampler are not time-inhomogenous Markov chains, it turns out that we will be able to apply these bounds to them through a conditioning argument. In Section \ref{SubsecNotCond} we further explain this important point and explain how the Equi-Energy algorithm, and many other adaptive algorithms, fit into our framework.
\end{remark}

\begin{remark}
Choosing
\be \label{EqChoiceOfLambda}
\lambda = \min \Big( \lambda_{\max},  \frac{ r \kappa T^{2}}{8 \sum_{\ell = 0}^{T-1} \E_{X_{T_b} \sim  \mu}[\mathcal{V}(X_{\ell})]} \Big)
\ee 
yields the optimal bound in Equation \eqref{InequalityConcBoundMain}. In particular, if $\lambda_{\max} \leq \frac{r \kappa T^{2}}{4 \sum_{\ell = 0}^{T-1} \E_{X_{T_b} \sim \mu}[\mathcal{V}(X_{\ell})]}$, we have the following bound:
\be \label{InequalityConcBoundAux}
\P_{\mu} \Big( \vert \hat{\pi}_{T,T_b}(f) - \E_{X_{T_b} \sim \mu}[\hat{\pi}_{T,T_b}(f)] \vert \geq r\Big) \leq 2 e^{-\frac{ r^{2} \kappa T^{2}}{16 \sum_{\ell = 0}^{T-1} \E_{X_{T_b} \sim \mu}[\mathcal{V}(X_{\ell})]}}.
\ee 

\end{remark}

\subsection{Preliminary results} In this section we gather some technical results which will be used in the proof of Theorem \ref{ThmConcAlmostPositiveAlt}. Although some of our proofs are quite similar to that in \cite{JoOl10}, these results are also of independent interest. 

\begin{lemma}\label{LemmaWassCont}

Fix a kernel $K$ with stationary distribution $\pi$ on the Polish space $(\Omega,d)$. Assume that $(\Omega,d)$ has diameter $D_{\Omega} < \infty$ and eccentricity $E(x) < \infty$ with respect to $\pi$. Assume that, for some $\kappa > 0$ and compact set $\mathcal{X} \subset \Omega$, $K$ satisfies
\be  \label{IneqSimpleWassContraction}
W_{d}( K(x,\cdot), K(y,\cdot)) \leq  (1 -\kappa)d(x,y)
\ee 
for all $x,y \in \mathcal{X}$. Let $\{K_{t} \}_{t \in \mathbb{N}}$ be a sequence of transition kernels such that, for some $\delta < \infty$ and $T \in \mathbb{N}$,
\be \label{IneqSimpWassCloseness}
\sup_{0 \leq t \leq T-1} W_{d}( K_{t}(x,\cdot), K(x,\cdot))< \delta
\ee 
for  all $x \in \mathcal{X}$. Let $\{ X_{t} \}_{t \in \mathbb{N}}$ be a time-inhomogenous Markov chain driven by the sequence of kernels $\{ K_{t} \}_{t \in \mathbb{N}}$ and started at $X_{0} = x$. Then:
\be \label{IneqAbsCoupResTwo}
W_{d}(\mathcal{L}(X_{T}), &\pi ) \leq
\frac{\delta}{\kappa} + (1 - \kappa)^{T} E(x) +  D_{\Omega} \sum_{t=0}^{T-1} \Big( K^{t+1}(x,\mathcal{X}^{c}) + \prod_{s=0}^{t} K_{s} (x, \mathcal{X}^{c}) \Big).
\ee 

\end{lemma}

\begin{proof}[Proof of Lemma \ref{LemmaWassCont}]
Applying the triangle inequality and subsequently inequalities \eqref{IneqSimpleWassContraction} and \eqref{IneqSimpWassCloseness} yields that for any $t \in \{ 0, 1, \ldots, T-1 \}$ and any $x, y \in \mathcal{X}$,
\be 
W_{d}(K(x,\cdot), K_{t}(y,\cdot)) &\leq W_{d}(K(x,\cdot), K(y,\cdot)) + W_{d}(K(y,\cdot), K_{t}(y,\cdot)) \\
&\leq (1-\kappa) d(x,y) + \delta. \label{IneqWassOneStepTri}
\ee 
Fix $\gamma > 0$. By the definition of Wasserstein distance, it is possible to couple the time-inhomogenous Markov chain $\{ X_{t} \}_{t \in \mathbb{N}}$ driven by the sequence of kernels $\{ K_{t} \}_{t \in \mathbb{N}}$ and started at $X_{0} = x$ to a time-homogenous Markov chain $\{Y_{t}\}_{t \in \mathbb{N}}$ driven by kernel $K$ and started at $Y_{0} \sim \pi$ so that $\E[d(X_{s+1}, Y_{s+1}) \vert X_{s}, Y_{s}] \leq W_{d}(K_{s}(X_{s}, \cdot), K(Y_{s}, \cdot)) + \gamma$.  Combining this with inequality \eqref{IneqWassOneStepTri}, we have:
\be 
W_{d}(\mathcal{L}(X_{T}), \pi) &\leq \E[d(X_{T}, Y_{T})] \\
&\leq \E[\delta + \gamma + (1-\kappa)d(X_{T-1}, Y_{T-1})] + D_{\Omega}(1 - \P[X_{T-1}, Y_{T-1} \in \mathcal{X}]).
\ee
Iterating this $T-1$ times, and noting that $Y_{0} \sim \pi$ and $X_{0} = x$ imply $\E[d(X_{0}, Y_{0})] = E(x)$, yields
\be
W_{d}(\mathcal{L}(X_{T}), \pi) &\leq \frac{\delta + \gamma}{\kappa} + (1-\kappa)^{T} E(x) + D_{\Omega} \sum_{t=0}^{T-1} \Big( K^{t+1}(x,\mathcal{X}^{c}) +\prod_{s=0}^{t} K_{s} (x, \mathcal{X}^{c}) \Big).
\ee 
Since this holds for all $\gamma > 0$, the claim follows.
\end{proof}

We record several simple bounds on higher powers of kernels for later reference:

\begin{lemma} [Continuity and Powers] \label{LemmContPow}
Fix $0 < \delta < \infty$, a kernel $K$ defined on a Polish space $(\Omega, d)$ with (possibly negative) curvature $\kappa$, and a collection of kernels 
$\{ K_{t} \}_{t \in \mathbb{N}}$ that satisfy 
\be 
\sup_{x \in \Omega} W_{d}(K_{t}(x,\cdot), K(x,\cdot)) < \delta
\ee 
for all $t \in \mathbb{N}$. Then we have for all $k, T \in \mathbb{N}$
\be 
W_{d}( (\prod_{t = T}^{T+k-1} K_{t}) (x,\cdot), (\prod_{t = T}^{T+k-1} K_{t})(y,\cdot)) &\leq 2 \delta \sum_{i=0}^{k-1} (1 - \kappa)^{i} + (1 -\kappa)^{k} d(x,y), \label{IneqContPow1} \\
W_{d}((\prod_{t = T}^{T+k-1} K_{t})(x,\cdot), K^{k}(y,\cdot)) &\leq \delta \sum_{i=0}^{k-1} (1 - \kappa)^{i} + (1 - \kappa)^{k} d(x,y). \label{IneqContPow2}
\ee
Furthermore, if $\kappa \geq 0$, for any Lipschitz function $f$ we have 

\be \label{IneqContPow4}
| ((\prod_{t = T}^{T+k-1} K_{t})f) (x) - ((\prod_{t = T}^{T+k-1} K_{t}) f)(y) | &\leq \| f \|_{\lip} \Big( \frac{4 \delta}{\kappa}+ (1 - \kappa)^{k} d(x,y) \Big)
\ee 
for all $k, T \in \mathbb{N}$. 
\end{lemma}

\begin{proof}
We begin by proving inequality \eqref{IneqContPow1} by induction on $k$. For $k=1$, it is trivial. Fix $\gamma > 0$. Let us assume that for some $k \geq 2$,
\be \label{eqn:indhyp}
W_d((\prod_{t = T}^{T+k-2} K_{t})(x,\cdot), (\prod_{t = T}^{T+k-2} K_{t})(y,\cdot)) \leq  2 \delta \sum_{i=0}^{k-2} (1 - \kappa)^{i} + (1 - \kappa)^{k-1} d(x,y).
\ee
By the above assumption, it is possible to couple two Markov chains $\{ X_{t} \}_{t=0}^{k-1}$ and $\{ Y_{t} \}_{t=0}^{k-1}$ started at $X_{0} =x$ and $Y_{0} = y$ respectively, both driven by the sequence of kernels $\{ K_{t} \}_{t \in \mathbb{N}}$, so that 
\be 
\E[d(X_{T + k-1}, Y_{T + k-1}) | X_{T} = x, Y_{T} = y] \leq \gamma +  2 \delta \sum_{i=0}^{k-2} (1 - \kappa)^{i} + (1 - \kappa)^{k-1} d(x,y).
\ee 
Then, we can couple $X_{T+k}$ and $Y_{T+k}$ so that
\be  
\E[d(X_{T+k}, Y_{T+k})] &\leq \gamma + \E \left[ W_{d}(K_{T+k-1}(X_{k-1}, \cdot), K_{T+k-1}(Y_{k-1}, \cdot)) \right] \\
&\leq \gamma + \E \Big[ W_{d}(K_{T+k-1}(X_{k-1}, \cdot), K(X_{k-1},\cdot)) \\
&\hspace{2cm}+ W_{d}(K(X_{k-1}, \cdot), K(Y_{k-1}, \cdot)) + W_{d}(K_{T+k-1}(Y_{k-1}, \cdot), K(Y_{k-1},\cdot)) \Big]\\
&\leq \gamma + 2 \delta + (1 - \kappa) \Big( \gamma + 2 \delta \sum_{i=0}^{k-2} (1 - \kappa)^{i} + (1 - \kappa)^{k-1} d(x,y) \Big) \\
&\leq (2 - \kappa) \gamma +  2 \delta \sum_{i=0}^{k-1} (1 - \kappa)^{i} + (1 - \kappa)^{k} d(x,y), \label{IneqCouplMainIneq}
\ee
where the penultimate line follows from the induction hypothesis \eqref{eqn:indhyp}. \par 

This implies that $W_{d}(X_{k}, Y_{k}) \leq (2 - \kappa) \gamma +  2 \delta \sum_{i=0}^{k-1} (1 - \kappa)^{i} + (1 - \kappa)^{k} d(x,y)$ for all $\gamma > 0$. Letting $\gamma$ go to 0, we conclude that 
\be 
W_{d}(X_{T+k}, Y_{T+k}) \leq 2 \delta + (1 - \kappa) \big( 2 \delta \sum_{i=0}^{k-2} (1 - \kappa)^{i} + (1 - \kappa)^{k-1} d(x,y) \big).
\ee
Since \eqref{eqn:indhyp} is trivially true for $k =1$, inequality \eqref{IneqContPow1} follows for all $k \in \mathbb{N}$ by induction. Inequality \eqref{IneqContPow2} can be proved similarly by induction on $k$; we omit the details.\par

To prove \eqref{IneqContPow4}, 
\be
| ((\prod_{t = T}^{T+k-1} K_{t})f) (x) - ((\prod_{t = T}^{T+k-1} K_{t}) f)(y) | &\leq \| f \|_{\lip} W_{d}(( \prod_{t = T}^{T+k-1} K_{t})(x, \cdot), (\prod_{t = T}^{T+k-1} K_{t})(y, \cdot)) \\
&\leq \| f \|_{\lip} \Big( \frac{4 \delta}{\kappa} + (1 - \kappa)^{k} d(x,y) \Big),
\ee 
where the last line is an application of inequality \eqref{IneqContPow1}. 

\end{proof}

\begin{cor} \label{CorContPow}
Fix $0 < \delta < \infty$, a kernel $K$ defined on a Polish space $(\Omega, d)$ with  curvature $\kappa > 0$ and stationary distribution $\pi$, and a collection of kernels 
$\{ K_{t} \}_{t \in \mathbb{N}}$ that satisfy 
\be 
\sup_{x \in \Omega} W_{d}(K_{t}(x,\cdot), K(x,\cdot)) < \delta
\ee 
for all $t \in \mathbb{N}$. Let $\{X_{t} \}_{t \in \mathbb{N}}$ be a time inhomogeneous Markov chain driven by $K_t$, with $X_{0} \sim \mu$ for some distribution $\mu$. Then for any 1-Lipschitz function $f$:
\be 
\Big| \E \Big( \frac{1}{T} \sum_{t=1}^{T} f(X_{t}) \Big) - \pi(f) \Big| \leq   \frac{2 \delta}{\kappa} + \frac{\E[E(X_{0})]}{\kappa T},
\ee 
where $E(x)$ is the eccentricity of the point $x$.
\end{cor}

\begin{proof}
Fix $\gamma > 0$. By inequality \eqref{IneqCouplMainIneq}, it is possible to couple a Markov chain $\{ Y_{t} \}_{t \in \mathbb{N}}$ with transition kernel $K$ and initial point $Y_{0} \sim \pi$ to $\{ X_{t} \}_{t \in \mathbb{N}}$ so that for all $t \leq T$,
\be 
\E[d(X_{t}, Y_{t})] \leq \gamma + 2 \frac{\delta}{\kappa} + (1 - \kappa)^{t} \E[d(X_{0},Y_{0})].
\ee  
It follows that
\be 
\Big| \E \Big( \frac{1}{T} \sum_{t=1}^{T} f(X_{t}) \Big) - \pi(f) \Big| &\leq \frac{1}{T} \sum_{t=1}^{T} \E[ d(X_{t}, Y_{t})] \\
&\leq \gamma +  \frac{2\delta}{\kappa} + \frac{1}{ \kappa T} \E[d(X_{0},Y_{0})] \\
&= \gamma +  \frac{2\delta}{\kappa} + \E[E(X_{0})]\frac{1}{\kappa T}.
\ee 
Letting $\gamma$ go to 0 gives the result.
\end{proof}

The following Lemma is a minor extension of Lemma 38 from \cite{Olli09}:
\begin{lemma}  
\label{Lem38Ana}
Fix a kernel $K$ defined on a Polish space $(\Omega, d)$ with granularity $\sigma_\infty$. Also fix constants $0 \leq A, B \leq 1$  and a function $\phi$ satisfying
\be \label{IneqNewABLip}
| \phi(x) - \phi(y) | \leq \max \left( A d(x,y), B \right)
\ee 
for all $x,y \in \Omega$. For $\lambda \leq \min\left( \frac{1}{3 A \sigma_{\infty}}, \frac{2}{3B} \right)$, we have
\be 
(K e^{\lambda \phi})(x) &\leq e^{\lambda K \phi(x)} \left(1 + \lambda^{2} \mathrm{Var}_{K(x,\cdot)} (\phi) \right) \leq e^{\lambda K \phi(x) + \lambda^{2} \mathrm{Var}_{K(x,\cdot)} (\phi) }.
\ee 
\end{lemma}

\begin{proof}

By exactly the same argument as in the proof of Lemma 38 of \cite{Olli09} we obtain, 
\be 
(K e^{\lambda \phi})(x) \leq e^{\lambda K \phi(x)} + \frac{\lambda^{2}}{2} \Big( \sup_{\mathrm{Supp}(K(x,\cdot))} e^{\lambda \phi(\cdot)} \Big) \mathrm{Var}_{K(x,\cdot)} (\phi).
\ee 
Since $\mathrm{diam} (\mathrm{Supp} (K(x,\cdot))) \leq 2 \sigma_{\infty}$ and $\phi$ satisfies inequality \eqref{IneqNewABLip}, we have
\be 
\sup_{\mathrm{Supp}(K(x,\cdot))} \phi \leq K \phi(x) + \max(B, 2 A \sigma_{\infty}).
\ee  
Combining these two inequalities gives
\be 
(K e^{\lambda \phi})(x) \leq e^{\lambda K \phi(x)} \left( 1 +  \frac{\lambda^{2}}{2} \left(  e^{\lambda \max(B, 2 A \sigma_{\infty}) } \right) \mathrm{Var}_{K(x,\cdot)} (\phi) \right).
\ee 
Since $e^{\lambda \max(B, 2 A \sigma_{\infty})} \leq 2$ by assumption, the result follows.
\end{proof}

The following analogue of Lemma 10 of \cite{JoOl10} is used to obtain exponential moment bounds for our Markov chain: 

\begin{lemma} 
\label{LemGenExpMomBound1}
Let $K$ be a kernel defined on a Polish space $(\Omega, d)$ and with curvature $\kappa > 0$. Fix $T_b, T \in \mathbb{N}$ and $0 < \delta < 1$, let $\mathcal{V}$ satisfy inequality \eqref{eqn:V}, and let $\{ K_{t} \}_{t=T_b}^{T_b + T}$ be a sequence of kernels satisfying
\be \label{IneqOldLemma6WassApprox} 
\sup_{T_b \leq t \leq T_b + T} \sup_{x \in \Omega} W_{d}(K(x,\cdot), K_{t}(x, \cdot)) \leq \delta.
\ee

For $0 < \delta < \delta_{\max}$, $0 < \lambda < \lambda_{\max}$ and $f$ any  $\frac{2}{\kappa T}$-Lipschitz function, \footnote{We slightly abuse notation and, when $t_{1} > t_{2}$, mean that $\prod_{s=t_{1}}^{t_{2}} K_{s} = \mathrm{Id}$, where $\mathrm{Id}$ is the identity kernel.}
\be \label{IneqGenExpMomConc} 
\Big( \prod_{s=T_b}^{T_b + t} K_{s} \Big)\big( e^{\lambda f} \big) \leq \exp \Big( \lambda \prod_{s=T_b}^{T_b + t} K_{s} f + \frac{4 \lambda^{2}}{\kappa T^{2} } \sum_{s=T_b}^{T_b + t } \big( \prod_{u=T_b}^{s-1} K_{u} \big) \mathcal{V}\Big)
\ee 
for all $0 \leq t \leq T$. 
\end{lemma} 

\begin{proof} 

Let $f$ be a $\frac{2}{\kappa T}$-Lipschitz function. For $0 < \lambda < \frac{\kappa T}{4\mathcal{C}_v}$ and $0 \leq t \leq T$, define the function
\be \label{eqn:gt}
g_{t} = \lambda \prod_{s=T_b}^{T_b + t} K_{s} f + \frac{4 \lambda^{2}}{\kappa T^{2} } \sum_{s=T_b}^{T_b + t} \big( \prod_{u=T_b}^{s-1} K_{u} \big) \mathcal{V}.
\ee 
Thus \eqref{IneqGenExpMomConc} can be re-written as
\be \label{eqn:gtrewr}
\Big( \prod_{s=T_b}^{T_b + t} K_{s} \Big)\big( e^{\lambda f} \big) \leq e^{g_t}.
\ee
We then define $g_{t}^{(1)}$ and $g_{t}^{(2)}$ to be the first and second terms in $g_{t}$ in \eqref{eqn:gt}. By inequality \eqref{IneqContPow4} from Lemma \ref{LemmContPow}, we have that
\be 
| g_{t}^{(1)}(x) - g_{t}^{(1)}(y) | \leq \max \Big( \lambda (2 (1 - \kappa)^{t})\frac{2}{\kappa T} d(x,y),  \lambda \frac{8 \delta}{\kappa} \frac{2}{\kappa T} \Big)
\ee 
and 
\be 
| g_{t}^{(2)}(x) - g_{t}^{(2)}(y) | \leq \max \Big(  \frac{4 \lambda^{2}}{\kappa T^{2}} \sum_{s=0}^{t} 2 (1 - \kappa)^{s} \mathcal{C}_{v} d(x,y), \frac{8 \lambda^{2}}{\lambda T^{2}} \frac{4 \delta (t+1)}{\kappa} \mathcal{C}_{v}  \Big).
\ee 
Combining these two inequalities and simplifying yields,
\be \label{EqLipConstOfExpBasic}
| g_{t}(x) - g_{t}(y) | \leq \max(A d(x,y), B_{t})
\ee 
for all $x,y \in \Omega$, where
\be \label{EqLipConstOfExp}
A =  \frac{2}{\kappa T},  B_t = \frac{t+1}{12 T}.
\ee 
We now prove inequality \eqref{IneqGenExpMomConc} by induction on $t$. The first step is to show that the conditions of Lemma \ref{Lem38Ana}  are satisfied by $g_{t}$ for all $t$. By inequality \eqref{EqLipConstOfExp}, condition \eqref{IneqNewABLip} of Lemma \ref{Lem38Ana} is satisfied for $\phi = g_{t}$ and $A = A, B= B_{t}$ as given in equation \eqref{EqLipConstOfExpBasic} and any $0 < \lambda < \lambda_{\max}$. By the same inequality and the definition of $\lambda_{\max}$ in Equation \eqref{eqn:lmax}, the condition on $\lambda$ is also satisfied. Thus, in the base case $t=0$, inequality \eqref{eqn:gtrewr} (and thus \eqref{IneqGenExpMomConc}) follows immediately from an application of Lemma \ref{Lem38Ana} and the bound on $\mathrm{Var} K_t(x,\cdot)$ given by inequality \eqref{eqn:V}. \\
We now assume that inequality \eqref{eqn:gtrewr} holds for all times $0 \leq t \leq w$ and show that it holds at time $w+1$. We calculate
\be 
\prod_{v = T_{b}}^{T_{b} + w + 1} K_{T_{b} +v } e^{\lambda f} &\leq K_{T_{b} + w + 1} e^{g_{w}} \leq e^{g_{w + 1}},
\ee 
where the first inequality is the induction hypothesis and the second inequality comes from applying Lemma \ref{Lem38Ana} with bound on variance given again by inequality \eqref{eqn:V}. Again, we have verified that the conditions of Lemma \ref{Lem38Ana} hold for $g_{s}$ for all $s$ and the claim follows from induction.
\end{proof}

\subsection{Proof of Theorem \ref{ThmConcAlmostPositiveAlt}} \label{SubsecFirstEeProof}
We now put together the results from the previous subsection
to prove Theorem \ref{ThmConcAlmostPositiveAlt}.
\begin{proof}[Proof of Theorem \ref{ThmConcAlmostPositiveAlt}]
Define
\be
\tilde{f}_{x_{1}, \ldots,\, x_{k-1}}(x_{k}) = k^{-1} \sum_{i=1}^{k} f(x_{i})
\ee for all $k \in \mathbb{N}$ and all $(x_{1}, \ldots, x_{k-1}) \in \Omega^{k-1}$, with the natural extension $\tilde{f}_{\emptyset}(x_{1}) = f(x_{1})$.  Also define the related function 
\be
\tilde{g}_{x_{1},\ldots,\,x_{k-1}}(x) = \tilde{f}_{x_{1},\ldots,x_{k-1}}(x) + \frac{ 4 \lambda}{\kappa T^{2}} \sum_{\ell = 0}^{T-k-1} \prod_{t = T_{b}}^{T_{b}+\ell - 1} K_{t} \mathcal{V}(x),
\ee
with the natural extension $\tilde{g}_{\emptyset}(x) = \tilde{f}_{\emptyset}(x_{1}) + \frac{ 4 \lambda}{\kappa T^{2}} \sum_{\ell = 0}^{T-1} \prod_{t = T_{b}}^{T_{b}+\ell - 1} K_{t} \mathcal{V}(x)$. Using the curvature assumption and  inequality \eqref{IneqOldLemma6WassApprox2}, we can apply inequality \eqref{IneqContPow4} from Lemma \ref{LemmContPow} to see that the function $\tilde{g}_{x_{1}, \ldots, x_{k-1}}(\cdot)$ satisfies $| \tilde{g}_{x_{1}, \ldots, x_{k-1}}(x) - \tilde{g}_{x_{1}, \ldots, x_{k-1}}(y)| \leq \max(Ad(x,y), B)$ where  $A = \frac{2 }{\kappa T}, B =  \frac{k}{12T}.$ 
The calculation of those constants is exactly the same as that used to find inequality \eqref{EqLipConstOfExp}. \par
Applying Lemma \ref{LemGenExpMomBound1} successively for $k = T_{b} + T-1, T_b + T-2, \ldots, T_{b}$ we find
\be
\E_{X_{T_{b}} \sim \mu}[e^{\lambda \hat{\pi}_{T,T_{b}}(f)}] 
&=\int_{\Omega^{T}} e^{\lambda \tilde{f}_{x_{1},\ldots,x_{T-1}}(x_{T})} K_{T_{b} +T-1}(x_{T-1},dx_{T}) \ldots K_{T_{b} + 1}(x_{1}, dx_{2}) \mu(dx_{1}) \\
&=\int_{\Omega^{T}} e^{\lambda \tilde{g}_{x_{1},\ldots,x_{T-1}}(x_{T})} K_{T_{b} +T-1}(x_{T-1},dx_{T}) \ldots K_{T_{b} + 1}(x_{1}, dx_{2}) \mu(dx_{1}) \\
&\leq \int_{\Omega^{T-1}} e^{\lambda \tilde{g}_{x_{1}, \ldots, X_{T-2}}(x_{T-1})  } K_{T_{b}+T-2}(x_{T-1}, dx_{T})\ldots K_{T_{b}+1}(x_{1}, dx_{2}) \mu(dx_{1})\\
&\leq \ldots \\
&\leq e^{\lambda \E_{X_{T_{b}} \sim \mu}[\hat{\pi}_{T,T_{b}}(f)]} e^{\frac{4 \lambda^{2}}{\kappa T^{2}} \sum_{\ell = 0}^{T-1} \E_{X_{T_{b}} \sim \mu}[\mathcal{V}(X_{T_{b} + \ell})]},
\ee
where the first two lines follow from definitions, and the remaining lines are repeated applications of Lemma \ref{LemGenExpMomBound1}. By Markov's inequality, this implies that for any $r>0$,
\be 
\P_{\mu} \Big( \vert \hat{\pi}_{T,T_b}(f) - \E_{X_{T_b} \sim \mu}[\hat{\pi}_{T,T_b}(f)] \vert \geq r\Big) \leq 2 e^{-\lambda r} e^{\frac{4 \lambda^{2}}{\kappa T^{2}} \sum_{\ell = 0}^{T-1} \E_{X_{T_{b}} \sim \mu}[\mathcal{V}(X_{T_{b} + \ell})]}
\ee 
and the proof is finished.
\end{proof}

\section{The Equi-Energy Sampler} \label{SecExEe}
The equi-energy (EE) algorithm introduced by \cite{KZW06} was inspired by the concept of microcanonical distribution in statistical mechanics. A key facet of this 
algorithm is that it facilitates moves within `energy levels' as explained below, so as to efficiently traverse between various modes of the target distribution. 
Our main motivation for studying this algorithm is the observation from \cite{KZW06} that the autocorrelation of samples taken from an equi-energy algorithm applied to a mixture of normal variables was quite a bit smaller than the autocorrelation of samples taken from a parallel tempering algorithm with the same target distribution. This can be seen empirically in the autocorrelation plots in Figures 3 and 4 of \cite{KZW06}, and is certainly suggested by the long moves in the sample paths shown in those figures. There are no rigorous non-asymptotic results explaining this empirical comparison. \par 

Our contribution here is two-fold. First, we study the equi-energy algorithm applied to a simple multi-modal target density on the unit circle. We compare the properties of this algorithm to a corresponding parallel tempering algorithm (see \cite{Geye91}) and to an underlying random walk Metropolis algorithm from which the equi-energy sampler is constructed.  For this example, using the concentration results obtained in the previous section, we rigorously show that the auto-covariance function of the equi-energy sampler decays \emph{faster} than that of the parallel tempering algorithm and the random walk Markov chain. These results thus corroborate the empirical observation made in \cite{KZW06}. Let us remark that, although the target density we consider is very simple, it encodes the key feature of multi-modality for which the equi-energy sampler is tailored. We believe that our results should also hold
for general multi-modal distributions; see the second point in Section \ref{SecDisc} for more details. \par
Next, we obtain finite sample bounds for the equi-energy sampler on a general class of target distributions via an inductive argument on the energy levels. This also yields the ergodicity of the equi-energy algorithm, and the convergence of the associated  kernels to a limiting kernel in Wasserstein distance. 

\subsection{The Algorithm}
We begin by recalling notation for the equi-energy and parallel tempering algorithms, and then define our main example.  We will generally follow the notation of \cite{KZW06} as closely as possible. \par 
Let  $\pi(x) \propto e^{-V(x)}$  be the target measure of interest on a metric space $(\Omega, d)$, and let $K_{\mathrm{MH}}$ be a reversible kernel on $\Omega$; we do not assume that the stationary measure of $K_{\mathrm{MH}}$ is $\pi$. We then fix a number of energy levels $\mathcal{K} \in \mathbb{N}$ and associated densities $\{ \pi_i \}_{0 \leq i \leq \mathcal{K}}$ with $\pi_{0} = \pi$. Next, we fix a constant $0 < p_{ee} < 1$ which defines the frequency with which the sampler takes special steps called \textit{equi-energy moves}. The equi-energy sampler itself consists of a collection of coupled stochastic process $\{ X_{t}^{(i)} \}_{t \in \mathbb{N}}$ at each energy level $0 \leq i \leq \mathcal{K}$, where chain $X_{t}^{(i)}$ is intended to (roughly) target the distribution $\pi_i$. The process $X_{t}^{(\mathcal{K})}$ is the simplest: it evolves according to the Metropolis-Hastings dynamics associated with proposal kernel $K_{\mathrm{MH}}$ and target distribution $\pi_{\mathcal{K}}$. We define $T_{b}^{(i)}$ to be the burn-in time for the $i^{\mathrm{th}}$ chain - that is, we wait until time $T_{b}^{(i)}$ before starting chain $X_{t}^{(i)}$. We generally set $T_{b}^{(\mathcal{K})} = 0$ and make no assumptions on the distribution of the starting values $X_{T_{b}^{(i)}}^{(i)}$. \\
To define the processes $\{X_{t}^{(i)}\}_{t \in \mathbb{N}}$ for $i < \mathcal{K}$, we require a function $H(v) = [h_{1}(v), h_{2}(v)]$ from $\mathbb{R}$ to intervals in $\mathbb{R}$. We then define the \textit{energy rings}:

\be \label{eqn:Ering}
\widehat{D}^{(i)}_{t, v} &= \{ X_{s}^{(i+1)} \, : \, T_{b}^{(i+1)} \leq s \leq  t, \, V(X_{s}^{(i+1)}) \in H(v) \}.
\ee
Based on these definitions, a generalization of the equi-energy sampler of \cite{KZW06} is given by Algorithm \ref{Alg:EE}.
\begin{algorithm} 
\caption{Equi-Energy Sampler for the Evolution of $X^{(i)}_t$}
\label{Alg:EE}
\begin{algorithmic} 
\REQUIRE Energy Ring: $\widehat{D}^{(i)}_{t, v}$ as given by \eqref{eqn:Ering}, Proposal Kernel
$K_{\mathrm{MH}}$.
\STATE $\bullet$ Simulate $p_t \in \mathrm{\{MH, EE\}}$ with
$$\P[p_{t} = \mathrm{EE}] = p_{ee}. $$
\IF{$p_t  = \mathrm{MH}$}
\STATE $\bullet$ Evolve $X^{(i)}_t$ via the kernel $K_\mathrm{MH}$ coupled with
the accept-reject mechanism, targetting distribution $\pi_{i}$.
\ELSE
\STATE $\bullet$ Make an \textit{Equi-Energy step} as follows. Simulate
\be
q^{(i)}_{t} & \sim \mathrm{Unif} \left(\widehat{D}^{(i)}_{t, V(X_{t}^{(i)})} \right), \, \quad w_{t}  \sim \mathrm{Unif}[0,1]
\ee
independently and set $r_t^{(i)} = r_{t}^{(i)}(X_{t}^{(i)},q^{(i)}_t)$ as given by \eqref{AccProbLocalH}.
\IF{$w_{t} < r_{t}^{(i)}$}
\STATE $\bullet$ Set $X_{t+1}^{(i)} = q^{(i)}_t $ 
\ELSE
\STATE $\bullet$ Set $X_{t+1}^{(i)} = X_{t}^{(i)}.$
\ENDIF
\ENDIF
\end{algorithmic}
\end{algorithm}

Thus the key idea behind the equi-energy proposal is to traverse between the modes while staying in regions where the target density  $\pi$ has (roughly) the same magnitude. Heuristically, we expect the equi-energy sampler to converge to the sampler with the same parameters $\{ \pi_{i} \}_{i=0}^{\mathcal{K}}$, $p_{ee}$, and $H$ that is described by Algorithm \ref{Alg:LEE}; this is referred to as the \textit{limiting algorithm}. The equi-energy step in Algorithm \ref{Alg:EE} is an \emph{ansatz} for that in Algorithm \ref{Alg:LEE}.

\begin{algorithm}[!ht] 
\caption{`Limiting' Sampler for the Evolution of $X^{(i)}_t$}
\label{Alg:LEE}
\begin{algorithmic} 
\REQUIRE Energy Ring: $\widehat{D}^{(i)}_{t, v}$ as given by \eqref{eqn:Ering}, Proposal Kernel
$K_{\mathrm{MH}}$.
\STATE $\bullet$ Simulate $p_t \in \mathrm{\{MH, EE\}}$ with
$$\P[p_{t} = \mathrm{EE}] = p_{ee}. $$
\IF{$p_t  = \mathrm{MH}$}
\STATE $\bullet$ Evolve $X^{(i)}_t$ via the kernel $K_\mathrm{MH}$ coupled with
the accept-reject mechanism, targetting distribution $\pi_{i}$.
\ELSE
\STATE $\bullet$ Make an \textit{Equi-Energy step} as follows. Simulate
\be
q_{t}^{(i)} &\sim \pi_{i+1}\vert_{\left \{ V^{-1}(H(V(X^{(i)}_{t}))) \right \}}
\ee
and set  $r_t^{(i)} = r_{t}^{(i)}(X_{t}^{(i)},q^{(i)}_t)$ as given by \eqref{AccProbLocalH}. Generate $w_{t} \sim \mathrm{Unif}[0,1]$.
\IF{$w_{t} < r_{t}^{(i)}$}
\STATE $\bullet$ Set $X^{(i)}_{t+1} = q_{t}^{(i)}$ 
\ELSE
\STATE $\bullet$ Set $X^{(i)}_{t+1}= X^{(i)}_{t}$
\ENDIF
\ENDIF
\end{algorithmic}
\end{algorithm}

\begin{remark}\label{rem:KZWH} In the original paper \cite{KZW06}, the authors defined the family of densities $\pi_i$ to be of the form
\be \label{EqStandardPiRep}
\pi_{i}(x) \propto e^{-\beta_i \max(H_{i}, V(x))},
\ee
for constants $\min_{x \in \Omega} V(x) = H_{0} < H_{1} < \ldots < H_{\mathcal{K}} < \infty$ and $1 = \beta_{0} > \beta_{1} > \ldots > \beta_{\mathcal{K}} \geq 0$.
Furthermore, in \cite{KZW06}, the intervals $H(v) = [h_1(v), h_2(v)]$ were defined by 
\be \label{EqStandardRingRep}
h_{1}(v) &= \sup \{ H_{j} \, : \, 0 \leq j \leq \mathcal{K}+1, \, H_{j} \leq v \} \\
h_{2}(v) &= \inf \{ H_{j} \, : \, 0 \leq j \leq \mathcal{K}+1, \, H_{j} > v \}.
\ee 
\end{remark} 

As in \cite{KZW06}, we choose $r_{t}^{(i)}$ to have the following property: if we sequentially choose $x \sim \pi_{i}$, $q \sim \pi_{i+1} |_{V^{-1}(H(V(x)))}$ and $u \sim \mathrm{U}[0,1]$, then the random variable $Y$ defined by
\begin{equation} \label{EeMoveRep2}
\begin{aligned}
Y &= q \, : \, U <  r_{t}^{(i)}(x,q), \\
Y &= x \, : \, U \geq r_{t}^{(i)}(x,q) \\
\end{aligned}
\end{equation}
should satisfy $Y \sim \pi_{i}$. Heuristically, the empirical measure associated with the set $\{ X_{t}^{(i+1)} \}^T_{t={T_b^{(i+1)}}}$ should be approximately $\pi_{i+1}$, and so our property is exactly the condition that the limiting kernel for the equi-energy moves is in fact a Metropolis-Hastings kernel with proposal distribution being the restriction of $\pi_{i+1}$ to an energy band and target distribution $\pi_{i}$.

We thus define $r_{t}^{(i)}$ to be 
\be 
r_{t}^{(i)}(x,q) &= \min \Big(1, {\pi_{i}(q) \over \pi_{i}(x)} {\pi_{i+1}(x) \over \pi_{i+1}(q)}{ \pi_{i+1}\big(V^{-1} (H(V(q)))\big) \over \pi_{i+1}\big(V^{-1} (H(V(x)))\big)} \Big). \label{AccProbLocalH}
\ee
In the original paper \cite{KZW06}, the choice of the function $H$ given in Equation \eqref{EqStandardRingRep} meant that ${ \pi_{i+1}\big(V^{-1} (H(V(q)))\big) \over \pi_{i+1}\big(V^{-1} (H(V(x)))\big)} =1$, and so that term vanishes from the acceptance ratio given in equation \eqref{AccProbLocalH}.
If the above ratio is difficult to calculate precisely, we can of course estimate the measure $\pi_{i+1}$ by the empirical measure associated with $\{ X_{t}^{(i+1)} \}$; it is possible to prove convergence of this `empirical' version of the equi-energy sampler by methods similar to those used to prove the convergence of the usual version.  The techniques in this paper are often easiest to apply when we choose the form
\be  \label{EqPropFunctDef}
H(v) = \left( v - \epsilon, v + \epsilon \right)
\ee 
for some $\epsilon > 0$, though it is certainly not required. The above definition agrees with the definition in Equation \eqref{EqStandardRingRep} for the example studied in Section \ref{SubsecEeVsPt}. For any targets $\{ \pi_{i} \}_{i \in \mathbb{N}}$, the choice $H(v) = (-\infty, \infty)$ trivially fits both Equation \eqref{EqPropFunctDef} and the form used in \cite{KZW06}.
\begin{remark} 
For continuous potentials $V(x)$, our modification of $H$ in \eqref{EqPropFunctDef} has the property that $d(x,y) \ll 1$ will imply that
\be
W_{d}(\mathrm{Unif}(\widehat{D}^{(i)}_{t,V(x)}), \mathrm{Unif}(\widehat{D}^{(i)}_{t,V(y)})) \ll 1
\ee 
holds with high probability as well.
This makes curvature-based analysis much simpler. The choice of $H$ in the original algorithm doesn't have this property on connected spaces, since for any $\delta > 0$, there may exist points $x,y$ with $d(x,y) < \delta$ and $H(V(x)) \cap H(V(y)) = \emptyset$. Generically, this results in limiting chains with curvature of $-\infty$. We point out that the example we study in Section \ref{SubsecEeVsPt} has a curvature of $-\infty$ for exactly this reason. As we illustrate there, this difficulty does not make curvature-based analysis impossible, as long as the curvature is positive on sufficiently large scales.
\end{remark}

\subsection{Notation and Conditionings} \label{SubsecNotCond}

We point out here some special properties of the equi-energy sampler that allow us to use results from Section \ref{SecConc} that apply only to Markov chains, and thus at first glance would seem inapplicable to our setting. First, conditioned on the history $\{X_{s}^{(j)}\}_{s \in \mathbb{N}, j \geq i+1}$, the process $\{ X_{t}^{(i)} \}$ is a time-inhomogenous Markov chain; that is,
\be 
\P[X_{t+1}^{(i)} \in A \vert \{ X_{s}^{(i)} \}_{s \leq t}, \, \{X_{s}^{(j)}\}_{s \in \mathbb{N}, j \geq i+1}] = \P[X_{t+1}^{(i)} \in A \vert X_{t}^{(i)}, \, \{X_{s}^{(j)}\}_{s \in \mathbb{N}, j \geq i+1}].
\ee 
Based on this observation, we note that Algorithm \ref{Alg:EE} implicitly defines the sequence of random kernels
\be 
K_{t}^{(i)}(x,A) \equiv \P[X_{t+1}^{(i)} \in A \vert X_{t}^{(i)} = x, \, \{X_{s}^{(j)}\}_{s \in \mathbb{N}, j \geq i+1} ],
\ee 
and furthermore $\mathcal{L}(X_{t+1}^{(i)} \vert \{ X_{s}^{(j)} \}_{s \leq t, g \geq i+1}, X_{T_{b}^{(i)}}^{(i)} = x) = \Big(\prod_{s=T_{b}^{(i)}}^{t} K_{s}^{(i)} \Big) (x, \cdot)$. This is important for our notation, as it means that Algorithm \ref{Alg:EE} defines a transition probability from all points $x \in \Omega$ at time $t$, not merely from $X_{t}^{(i)}$.  As suggested by this discussion, throughout our analysis of the equi-energy process at level $i$ we will condition on the entire history of the chain $\{X_{t}^{(j)} \}_{t \in \mathbb{N}, j \geq i+1}$ at higher temperatures. For instance, we will write $X_{t+1}^{(i)} \sim \mu \left( \prod_{s=1}^{t} K_{s}^{(i)} \right)$ when $X_{0}^{(i)} \sim \mu$ as a shorthand for denoting the conditional distribution of $X_{t+1}^{(i)}$ given $X_{0}^{(i)} \sim \mu$ \textit{and} the chains $\{X_{t}^{(j)} \}_{t \in \mathbb{N}, j \geq i+1}$. Similarly, we write $\P_{i}[\cdot] = \P[\cdot \vert  \{X_{t}^{(j)} \}_{t \in \mathbb{N}, j \geq i+1} ]$.

\subsection{Parallel Tempering} In this section we define the notation for a parallel tempering algorithm  \cite{Geye91} for sampling from the target distribution $\pi(x) \propto e^{-V(x)}$. We again fix a number of chains $\mathcal{K}$, a sequence of kernels $\{ \pi_{i} \}_{0 \leq i \leq \mathcal{K}}$ (generally, these are defined by $\pi_{i}(x) \propto e^{-\beta_{i} V(x)}$ for some decreasing sequence of inverse temperatures $\{ \beta_{i} \}_{0 \leq i \leq \mathcal{K}}$), and an underlying proposal kernel $K_{\mathrm{MH}}$. We also fix a switching probability, which (abusing notation slightly) is denoted by $p_{ee}$. To run the algorithm, we begin $\mathcal{K}+1$ chains with initial conditions $X_{0}^{(\mathcal{K})}, \ldots, X_{0}^{(0)}$. For the comparison with the equi-energy sampler given below, we use $\mathcal{K} = 1$, and so we explain only this case in detail (see, \textit{e.g.}, \cite{Geye91} for the general case). 
\begin{algorithm}[!ht] 
\caption{Parallel tempering with two temperatures for the evolution of $X_t^{(i)}$, $i \in \{0,1\}$}
\label{Alg:PT}
\begin{algorithmic} 
\REQUIRE Proposal Kernel $K_{\mathrm{MH}}$; inverse temperatures $\beta_0,\beta_1$.
\STATE  $\bullet$ Draw $u \sim \mathrm{Unif}[0,1]$.
\STATE $\bullet$ Simulate $p_t \in \mathrm{\{MH, EE\}}$ with
$$\P[p_{t} = \mathrm{EE}] = p_{ee}. $$
\IF{$p_t  = \mathrm{MH}$}
\STATE $\bullet$ For $i \in \{0,1\}$, evolve $X^{(i)}_t$ via the kernel $K_\mathrm{MH}$ coupled with
the accept-reject mechanism, targetting distribution $\pi_{i}$.
\ELSE
 \IF{$u < \min \Big( 1, e^{-(V(X_{t}^{(0)}) - V(X_{t}^{(1)}) ) (\beta_0 - \beta_1) } \Big)$}
\STATE $\bullet$ Make a swap move: \be
X_{t+1}^{(0)} &= X_{t}^{(1)}, \quad
     X_{t+1}^{(1)}  = X_{t}^{(0)}. 
     \ee
\ELSE
\STATE $\bullet$ Set $X_{t+1}^{(i)} = X_{t}^{(i)}$, $i \in \{0,1\}$.
\ENDIF
\ENDIF
\end{algorithmic}
\end{algorithm}

\subsection{Equi-Energy Sampler \textit{vs.} Parallel Tempering} \label{SubsecEeVsPt}

\begin{figure}[h] 
\begin{tikzpicture}
\draw[->] (0,0) -- (11,0) node[anchor=north] {$x$};
\draw	(0,0) node[anchor=north] {0}
		(1,0) node[anchor=north] {$\frac{1}{2M}$}
		(2,0) node[anchor=north] {$\frac{1}{M}$}
		(9,0) node[anchor=north] {$1 -\frac{1}{2M} $}
		(10,0) node[anchor=north] {$1$};

\draw	(-0.5,1) node{{$\frac{2}{1+e^H}$}}
                 (-0.5,2) node{{$\frac{2e^H}{1+e^H}$}};

\draw[->] (0,0) -- (0,3) node[anchor=east] {$\pi(x)$};
\draw [fill] (0,0) circle [radius=.05];
\draw [fill] (1,0) circle [radius=.05];
\draw [fill] (2,0) circle [radius=.05];
\draw [fill] (0,1) circle [radius=.05];
\draw [fill] (0,2) circle [radius=.05];
\draw [fill] (9,0) circle [radius=.05];
\draw [fill] (10,0) circle [radius=.05];
\draw[thick] (0,0) -- (0,1) -- (1,1) -- (1,2) -- (2,2) --(2,1)--(3,1) --(3,2) --(4,2)--(4,1) ;
\draw[thick] (10,0) -- (10,2) -- (9,2) -- (9,1) --(8,1) --(8,2) -- (7,2)--(7,1) ;
\draw[dotted] (4,1) -- (5,1);
\draw[dotted] (0,2) -- (1,2);
\draw[dotted] (7,1) -- (6,1);
\end{tikzpicture}
\caption{Target density for the equi-energy sampler \textit{vs.} parallel tempering
comparison}
\label{fig:EEPT}
\end{figure}
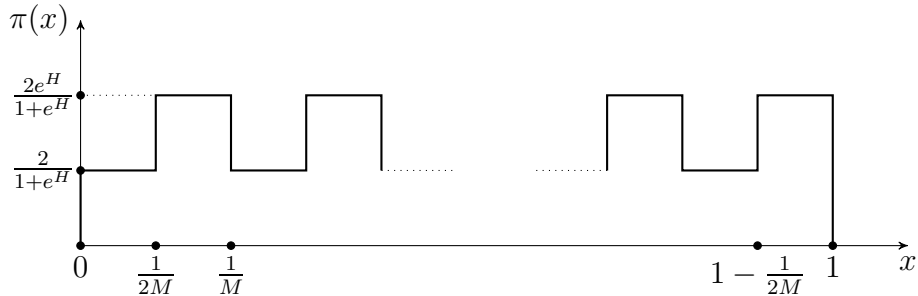

In this section we compare the performances of the equi-energy sampler and the parallel tempering
algorithm for a simple target distribution $\pi$ on the unit circle, as shown in Figure \ref{fig:EEPT}. 
Take $\Omega$ to be the unit circle, \textit{i.e.}, $[0,1]$ with $0$ and $1$ identified. We define the usual `embedded' metric on the unit circle:
\be \label{DefDistOnTorus}
d(x,y) = \min( | x - y|, 1 - |x - y|).
\ee 
 Fix a number of wells $M > 1$ and a well depth $H \geq 0$. The target distribution $\pi$ on $[0,1]$ has density
\begin{equation} \label{EqDefOfSquareToothPot}
\begin{aligned}
\pi(x) &= \frac{2}{1 + e^{H}} \, \, : \, 0 \leq x \leq \frac{1}{2M}, \\
\pi(x) &= \frac{2e^{H}}{1 + e^{H}} \, \, : \, \frac{1}{2M} < x < \frac{1}{M} \\
\end{aligned}
\end{equation} 
and repeating with period $\frac{1}{M}$, so that $\pi(x) = \pi(x - M^{-1})$ for all $x \in [0,1]$. This has the associated potential $V(x) \in \{0, H \}$. Fix a diffusion rate $0 < c < \frac{1}{4M}$ and define the  proposal kernel 
\be \label{eqn:Kmet}
K_{\mathrm{MH}}(x,A) = \frac{1}{2c} \lambda(A \cap \{ y \in \Omega \, : \, d(x,y) < c \}),
\ee
where $\lambda$ is Lebesgue measure on $[0,1]$. 

For both the equi-energy and parallel tempering algorithms, we will use $K_{\mathrm{MH}}$ as the underlying Metropolis-Hastings proposal kernel. In both cases we will use $\mathcal{K} = 1$. For the equi-energy sampler, we will follow the definitions of $\pi_{i}$ and $H$ given in equations \eqref{EqStandardPiRep} and \eqref{EqStandardRingRep}, using parameters $H_{0} = 0$, $H_{1} = H$, $\beta_{0} = 1$ and $\beta_{1} = 0$. Thus for both samplers, our first chain will target $\pi_{0} = \pi$ and our second chain will target $\pi_{1} = \lambda$, the Lebesgue measure on $\Omega$. We will leave free the swap probability $p_{ee}$, and state results in terms of this probability. Thus in this example we have two energy rings $\widehat{D}_{t,0}^{(0)}, \widehat{D}_{t,H}^{(0)}$ corresponding to $V = 0, H$ respectively.  
\subsection{Main Results}
We describe the convergence results for the equi-energy and parallel tempering samplers. For the equi-energy sampler, assume $X_{0}^{(1)} \sim Q$, where $Q$ has continuous density with respect to Lebesgue measure and
\be
\left \| \frac{d Q}{d \lambda} \right \|_{2} = N_{Q} < 2.
\ee 
Here $\|\cdot \|_2$ denotes the $L_2(\Omega, d\lambda)$ norm.

Write
\be 
A_{1} &= \min \left( \frac{p_{ee}(1 - p_{ee}) M}{32}, \frac{p_{ee}}{2c} \right)\\
A_{2} &= \frac{16 eH}{Ac} e^{-H}.
\ee 
Our first quantitative result shows that, after the burn-in period and a reasonable running time, the equi-energy kernel has small autocovariance with high probability: 

\begin{thm} [Intermediate-Time Autocorrelations for Equi-Energy Samplers] \label{ThmIntAutoEE}
Let the metric $d$ be defined on the state space $\Omega = [0,1]$ as in \eqref{DefDistOnTorus}, let $Q$ be a distribution on $\Omega$ that has continuous density with respect to Lebesgue measure and satisfies $ \Big \|{dQ \over d\lambda } \Big \|_2 \leq 2$. Fix a 1-Lipschitz function $f$ on $\Omega$ with $\E_{\pi}[f] = 0$, and also fix $k \in \mathbb{N} \cup \{ \infty \}$, $0 < \epsilon < \min \left( \frac{1}{512}, \frac{cM}{384} \right)$, and burn-in time $T_{b}^{(0)} \equiv T >  110592 c^{-2} \left(\epsilon^{-1} + 1 \right)^{2} \log \left( 4608 \left(\epsilon^{-1} + 1 \right)^{2} \right) $. \\

For $X_{0}^{(1)} \sim Q$ and $X_{T}^{(0)} = x \in \Omega$, we have

\be \label{eqn:autocovee}
\E[f(X^{(0)}_{T})f(X^{(0)}_{T+S}) \vert \{ X_{s}^{(1)}\}_{s \in \mathbb{N}}  ] \leq 2 e^{-A_{1} c \lfloor \frac{S}{4} \rfloor} +  2 \epsilon + A_{2} 
\ee
holds simultaneously for all $0 \leq S \leq T \left(1 + \frac{k \epsilon}{2} \right)$ with probability at least $1 - \frac{ \epsilon }{12} \min \left( k , \frac{1}{1 - \epsilon^{\frac{1}{2} \epsilon}} \right)$.
\end{thm}

\begin{remark}
Note that in Theorem \ref{ThmIntAutoEE}, we assume that $\E_{\pi}[f] = 0$, not that $\E \Big[ \sum_{s = T}^{T+S} f ( X^{(0)}_{s}) \Big] = 0$. In particular, this means that the concluding inequality \eqref{eqn:autocovee} is a bound on both the variance and bias of the equi-energy sampler. We also note that $\lim_{\epsilon \rightarrow 0} \frac{\epsilon}{1 - \epsilon^{\frac{1}{2} \epsilon}} = 0$, and so choosing $k = \infty$ gives a bound that holds with probability tending to 1 as $\epsilon $ goes to 0.
\end{remark}

From this bound, it is straightforward to see that the Equi-Energy sampler is much more efficient than its underlying Markov chain when the barrier $H$ between modes is large. The following bound gives one formalization of this comparison:

\begin{thm} [Mixing Time of Underlying Markov Chain]
Assume $M \geq 2$. The Metropolis-Hastings chain with proposal kernel $K_{\mathrm{MH}}$ as defined in Equation \eqref{eqn:Kmet} and target distribution $\pi$ has mixing time satisfying
\be 
\tau_{\mathrm{mix}} = \Om \left(e^{H} \right).
\ee 
\end{thm}
\begin{proof}
Define the bottleneck constant of a measurable set $A \subset \Omega$ by
\be 
\Phi(A) = \frac{\int_{x \in A} K_{\mathrm{MH}}(x,A^{c}) dx}{\pi_{0}(A)},
\ee 
and the Cheeger constant of the kernel by
 $\Phi = \inf_{\pi_{0}(A) \leq \frac{1}{2}} \Phi(A).$ 
Clearly,
\be 
\Phi &\leq \Phi([0, \frac{1}{2M}]) \leq 8 M e^{-H}.
\ee 
Combining this with Theorem 7.6 of \cite{LPW09} (stated there for discrete state space, but applicable in this setting with minor changes), we have
\be 
\tau_{\mathrm{mix}} &= \Om \left( \frac{1}{\Phi} \right) = \Om \left( e^{H} \right)
\ee 
proving the claim.
\end{proof}

To compare the equi-energy sampler to the parallel tempering algorithm, we need a lower bound on the convergence rate of the latter; the following easy bound suffices.

\begin{thm} [Intermediate-Time Autocovariance for Parallel Tempering Samplers] \label{ThmIntAutoPT} Assume $M \geq 2$ is even. Fix $T > 0$ and let $\{ (X_t^{(0)}, X^{(1)}_t) \}_{t \in \mathbb{N}}$ denote
the parallel tempering chain. Define the event
\be
 \mathcal{S} = \left \{ \frac{1}{8M} \leq X_{T}^{(0)}, X_{T}^{(1)} \leq \frac{3}{8M} \right \}.
 \ee
  Then, for any $S> 0$, we have the lower bound
\be \label{IneqIntAutPtMainBd}
\E[(X_{T}^{(0)} - \pi_{0}[X_{T}^{(0)}]) (X_{T+S}^{(0)} - \pi_{0}[X_{T+S}^{(0)}]) \vert \mathcal{S}] \geq \frac{1}{32} - e^{- \frac{1}{1024 M^{2} c^{2} S}}.
\ee
\end{thm}

\begin{proof}[Proof of Theorem \ref{ThmIntAutoPT}] 
Assume $\frac{1}{8M} \leq X_{T}^{(0)}, X_{T}^{(1)} \leq \frac{3}{8M}$ and let $\{ x_{s} \}_{s \in \mathbb{N}}$ be an i.i.d. sequence of random variables distributed as $\mathrm{Unif}[-c,c]$. We can then write, for any $r < \frac{1}{16M}$,
\be 
\P\Big[\sup_{T \leq t \leq T+S, \,i \in \{0,1\}} &\inf_{j \in \{0,1\}} \vert X_{t}^{(i)} - X_{T}^{(j)} \vert > 2r\Big] \leq \P \Big[\sup_{T \leq t \leq T+S} \vert \sum_{s=T}^{t} x_{s} \vert > r \Big].
\ee 

By Doob's maximal inequality,
\be 
\P\Big[\sup_{T \leq t \leq T+S} \vert \sum_{s=T}^{t} x_{s} \vert > r\Big] \leq 4 e^{-\frac{r^{2}}{4 S c^{2}}}.
\ee 
Thus, 
\be 
\P\Big[X_{T+S}^{(0)}, X_{T+S}^{(1)} \leq \frac{1}{2M} \, \big \vert \, \frac{1}{8M} \leq X_{T}^{(0)}, X_{T}^{(1)} \leq \frac{3}{8M}\Big] \geq 1 - e^{- \frac{1}{1024 M^{2} c^{2} S}}.
\ee 
We then have
\be 
\E[(X_{T}^{(0)} - \pi_{0}[X_{T}^{(0)}]) (X_{T+S}^{(0)} - \pi_{0}[X_{T+S}^{(0)}]) \vert \mathcal{S}] &= \E[(X_{T}^{(0)} - \frac{1}{2}) (X_{T+S}^{(0)} - \frac{1}{2}) \vert \mathcal{S}] \\
&= \E[X_{T}^{(0)} X_{T+S}^{(0)} | \mathcal{S}] - \frac{1}{2} \E[X_{T}^{(0)}  | \mathcal{S}] - \frac{1}{2} \E[ X_{T+S}^{(0)} | \mathcal{S}] + \frac{1}{4} \\
&\geq 0 - \frac{1}{2} \frac{3}{8M} -  \frac{1}{2} \Big( \big(1 - e^{- \frac{1}{1024 M^{2} c^{2} S}} \big) \frac{1}{2 M} +e^{- \frac{1}{1024 M^{2} c^{2} \sqrt{S}}} \Big) + \frac{1}{4} \\
&\geq \frac{1}{32} - e^{- \frac{1}{1024 M^{2} c^{2} S}}.
\ee 
This completes the proof.
\end{proof}

\begin{remark}

The mixing time of the parallel tempering algorithm is also $ \Om \left( \frac{c^{-2}}{\log \left( c^{-1} \right)^{2}} \right)$; in the regime $c^{-1} \approx M$, this is an improvement on Theorem \ref{ThmIntAutoPT}. One (lengthy) proof of this bound can be found by applying the mixing lower bound of \cite{GMT06} for a discretized version of the walk and then using a limiting argument similar to that in our supporting document.
\end{remark}

\subsection{Comparison of equi-energy \textit{vs.} Parallel tempering}
Before proving Theorems \ref{ThmIntAutoEE} and \ref{ThmIntAutoPT}, we explain our interest in them: showing that the equi-energy sampler gives better results than the parallel tempering sampler after some realistic burn-in time $T$. In particular, when $c$ is small these calculations show that the autocovariance of the equi-energy sampler decays at a rate of at least roughly $cM$, while the autocovariance of the parallel tempering sampler decays at rate of at most roughly $c^{2} M^{2}$; when $c^{-1} \gg M$, the former decay rate is much faster than the latter. \par

We point out one aspect of Theorem \ref{ThmIntAutoEE} that may be surprising to those familiar with the MCMC literature: our bound on autocovariance does not decay exponentially for all values of $S$. This property turns out to be unavoidable. As we discuss in more detail in Section \ref{SubsecFinVsInfAut} for a similar example, the autocorrelation decays only like $\frac{1}{T+S}$ in general - that is, $\mathrm{Var}[\widehat{\pi}_{S,T}(f)] = \Om\left( \frac{1}{T+S} \right)$. Thus, a bound on the correlation that decays exponentially in $S$ for all $S > 0$ is impossible.  

\subsection{Proof of Theorem \ref{ThmIntAutoEE}}
Our proof proceeds by first showing that for sufficiently large times $T$, the measure $\mathrm{Unif}(\{ X_{t}^{(1)} \}_{0 \leq t \leq T} )$ is close to $\mathrm{Unif}[0,1]$, and then using this to show that the evolution of the stochastic process $\{ X_{t}^{(1)} \}_{t > T}$ is very close to that described by Algorithm 2. \\
Recall that $\Omega = [0,1]$ and that the associated metric $d$ is given by Equation \eqref{DefDistOnTorus}.
We define the \textit{quantile} coupling of any distribution $\mu$ with density $\rho$ uniformly bounded below on $[0,1]$ and distribution $\nu$ with either atomic measure or density uniformly bounded below, as follows. First, choose $Y \sim \mu$. Then set
\be \label{QuantCoupDef}
X = \inf \{x \in \mathrm{Supp}(\nu) \, : \, \nu([0,x]) > \mu([0,Y]) \}.
\ee  

It is easy to check by standard Fourier analysis results (see, \textit{e.g.}, \cite{Rose93}) that for $c < \frac{1}{8}$, the relaxation time of the kernel $K_\mathrm{MH}$ given in \eqref{eqn:Kmet} is bounded by 
\be \label{IneqSpecGap}
\frac{1}{8c^{2}} \leq \tau_{\mathrm{rel}} \leq \frac{8}{c^{2}}.
\ee
This implies that the spectral gap of the top chain is on the order of $c^{2}$. Before proceeding, we need the following result of Lezaud \cite{Leza98} on the medium deviation of a Markov chain on a finite state space with spectral gap (see Remark 3 of \cite{Leza98} for this restatement of the result): 
\begin{thm}[\cite{Leza98}, Theorem 1.1] \label{ThmLezaud}
Fix a Markov chain $\{ Y_{i} \}_{i \in \mathbb{N}}$ on a finite state space with initial measure $Y_{0} \sim Q$, invariant measure $\nu$ and spectral gap $\lambda_{1}$. Also fix a function $f$ with $\nu(f) = 0$ and $ \nu \left( f^{2} \right) \leq 1$. Then for all $\gamma > 0$,
\be
\P \Big[n^{-1} \sum_{i=1}^{n} f(Y_{i}) \geq \gamma \Big] \leq \Big \|{dQ \over d\nu } \Big \|_2 \, \exp \Big( \frac{\lambda_{1}}{5} - \frac{ n \gamma^{2} \lambda_{1} }{12 }\Big).
\ee

\end{thm}

Next, denote by $d_{\mathrm{LP}}(\mu, \nu)$ the Levy-Prokhorov distance between two distributions $\mu$ and $\nu$; if they are both distributions on $[0,1]$, this distance can be defined as:
\be
d_{\mathrm{LP}}(\mu, \nu) = \inf \{ \epsilon > 0\, : \, \forall \, x \in [0,1], \, \mu([0,x-\epsilon]) - \epsilon \leq \nu([0,x]) \leq \mu([0,x+\epsilon]) + \epsilon  \}.
\ee
Let 
\be
G = \Big \{ x \in [0,1] \, : \, \inf_{i \in \mathbb{N}} | x - \frac{2i - 1}{2M} | < \frac{c}{16} \Big \}.
\ee
 We then define the slightly modified distance $\tilde{d}_{\mathrm{LP}}$ on the space of measures on $[0,1]$ by:
\be 
\tilde{d}_{\mathrm{LP}}(\mu,\nu) = \max(d_{\mathrm{LP}}(\mu,\nu), | \mu(G) - \nu(G) | ).
\ee 

We can now show:

\begin{lemma} [Convergence of $\widehat{D}_{t,0}^{(0)}$] \label{QuantCoupIneq}
Fix $\epsilon^{-1} \in \mathbb{N}$ with $\epsilon^{-1} > 256$, $T > \frac{192(\epsilon^{-1} + 1)^{2}}{ c^{2}} \log \left( 4 (\epsilon^{-1} + 1)^{2} \right)$ and $0 \leq \alpha < \infty$. Assume that $X_{0}^{(1)}$ is distributed to some distribution $Q$ with continuous density with respect to Lebesgue measure that satisfies $  \Big \|{dQ \over d\lambda } \Big \|_2 \leq 2$. Then for $t \geq (1 + \alpha)T$, 
\be
W_d(\mathrm{Unif}(\widehat{D}_{t,0}^{(0)}), \mathrm{Unif}(V^{-1}\{0\}))  \leq 2 \tilde{d}_{\mathrm{LP}} (\mathrm{Unif}(\widehat{D}_{t,0}^{(0)}), \mathrm{Unif}(V^{-1}\{0\})) \leq 12 \epsilon
\ee
with probability at least $1 - \epsilon^{1 + \alpha}$. 
\end{lemma}
\begin{proof}
The proof relies on applying Theorem \ref{ThmLezaud} to $X_{t}^{(1)}$ with some functions of the form $f_{x}(y) = \textbf{1}_{0 \leq y \leq x} - x$ as well as the function $f_{G}(y) = \textbf{1}_{y \in G} - \pi_{1}(G)$. Define $\mathcal{I}_{\epsilon}$ to be the collection of functions
\be
\mathcal{I}_{\epsilon} = \{ f_{G} \} \cup \{ f_{i \epsilon} \}_{i \in \mathbb{N}, 1 \leq i \leq \epsilon^{-1}}. 
\ee

By direct computation, it is easy to check that 
\begin{equation}  \label{IneqAsymVar}
\pi_{1}( f^{2})  \leq 1
\end{equation}
for all $f \in \mathcal{I}_{\epsilon}$.
By Lemma 1.1 in our supporting document, we can apply Theorem \ref{ThmLezaud} as stated to our example (the bound on the spectral gap in inequality \eqref{IneqSpecGap} also holds for the discrete chains used in that Lemma; the calculations in Example 12.3.1 of \cite{LPW09} give this bound with minor modification).
Using inequalities \eqref{IneqSpecGap} and \eqref{IneqAsymVar} with Theorem \ref{ThmLezaud} yields that for $f \in \mathcal{I}_{\epsilon}$, $t \in \mathbb{N}$ and $\gamma > 0$, 
\be \label{IneqLezApp}
\P\Big[ \Big| \frac{1}{t+1} \sum_{s=T_{b}}^{T_{b}+ t} f(X_{s}^{(1)}) \Big| > \gamma \Big] \leq 4 \exp \Big( - t c^{2} \big( \frac{\gamma^{2}}{96} - \frac{8}{5t} \big) \Big).
\ee
By Equation \eqref{IneqLezApp} and a union bound, we have
\be \label{IneqEeCompPreLevProk}
\P \Big[ \sup_{f \in \mathcal{I}_{\epsilon}} \Big| \frac{1}{t+1} \sum_{s=T_{b}}^{T_{b} + t} f(X_{s}^{(1)}) \Big| > \epsilon \Big] &\leq 4 (\epsilon^{-1} + 1) \exp \Big( - t c^{2} \big( \frac{\epsilon^{2}}{96} - \frac{8}{5t} \big) \Big) 
\leq \epsilon^{1+ \alpha}.
\ee
Let $X \sim \mathrm{Unif}(\widehat{D}_{t,0}^{(0)})$ and let $Y \sim \mathrm{Unif}(V^{-1}\{0\})$. We note that, for $1 \leq i \leq \epsilon^{-1}$,
\be 
| \P_{0}[X \leq i \epsilon] - \P[Y \leq i \epsilon] | = \Big| \frac{1}{t+1} \sum_{s=T_{b}}^{T_{b} + t} f_{i \epsilon}(X_{s}^{(1)}) \Big| \leq  \sup_{f \in \mathcal{I}_{\epsilon}} \Big| \frac{1}{t+1} \sum_{s=T_{b}}^{T_{b} + t} f(X_{s}^{(1)}) \Big|,
\ee 
and also that
\be \label{IneqEeTildeDPL}
| \P_{0}[X \in G] - \P[Y \in G] | = \Big| \frac{1}{t+1} \sum_{s=T_{b}}^{T_{b} + t} f_{G}(X_{s}^{(1)}) \Big| \leq  \sup_{f \in \mathcal{I}_{\epsilon}} \Big| \frac{1}{t+1} \sum_{s=T_{b}}^{T_{b} + t} f(X_{s}^{(1)}) \Big|.
\ee
For $x \in [0,1]$, define $x^{+} = \epsilon \, \inf \{ i \in \mathbb{N} \, : \, i \epsilon \geq x \}$ and $x^{-} =  \epsilon \,  \sup \{ i \in \mathbb{N}  \, : \, i \epsilon \leq x \}$. We then have
\be 
d_{\mathrm{LP}}(X,Y) &\leq \sup_{x \in [0,1]} | \P_{0}[X \leq x] - \P[Y \leq x] | \\
&\leq \sup_{x \in [0,1]} \left( | \P_{0}[X \leq x] - \P_{0}[X \leq x^{+}] | + | \P_{0}[X \leq x^{+}] - \P[Y \leq x^{+}] | + | \P[Y \leq x^{+}] - \P[Y \leq x] | \right) \\
&\leq  \sup_{x \in [0,1]} \left( | \P_{0}[X \leq x^{-}] - \P_{0}[X \leq x^{+}] | + | \P_{0}[X \leq x^{+}] - \P[Y \leq x^{+}] | + | \P[Y \leq x^{+}] - \P[Y \leq x] | \right) \\
&\leq \sup_{x \in [0,1]} \Big( \left( | \P_{0}[X \leq x^{-}]  - \P[Y \leq x^{-}] | +| \P[Y \leq x^{-}]  - \P[Y \leq x^{+}] | + | \P[Y \leq x^{+}]  - \P_{0}[X \leq x^{+}] | \right) \\
&\hspace{2cm}+  | \P_{0}[X \leq x^{+}] - \P[Y \leq x^{+}] | + | \P[Y \leq x^{+}] - \P[Y \leq x] | \Big)\\
&\leq 3 \epsilon + 3 \sup_{f \in \mathcal{I}_{\epsilon}} \Big| t^{-1} \sum_{s=T_{b}+1}^{T_{b} + t} f(X_{s}^{(1)}) \Big|.
\ee 

Conditioned on the event that $\Big \{ \sup_{f \in \mathcal{I}_{\epsilon}} \frac{1}{t+1} \sum_{s=T_{b}}^{T_{b} + t} f(X_{s}^{(1)}) < \epsilon \Big \},$ this bound combined with inequality \eqref{IneqEeCompPreLevProk} and inequality \eqref{IneqEeTildeDPL} implies that
\be \label{IneqProkhStuffMainRes}
\tilde{d}_{\mathrm{LP}} (\mathrm{Unif}(\widehat{D}_{t,0}^{(0)}), \mathrm{Unif}(V^{-1}\{0\})) \leq 6 \epsilon.
\ee

By Theorem 2 of \cite{GiSu07}, 
\be 
W_{d}(X,Y) \leq 2 d_{\mathrm{LP}}(X,Y)  \leq 2 \tilde{d}_{\mathrm{LP}}(X,Y).
\ee 
Combining this with inequality \eqref{IneqProkhStuffMainRes} completes the proof.
\end{proof}

\begin{remark} \label{RmkCoupExtraPts}
We point out the trivial inequality that, whenever $F$ and $G$ are the empirical distributions of two sets $S_{F} = \{x_{1}, \ldots, x_{n} \}$ and $S_{G} = \{ x_{1}, \ldots, x_{n}, y_{1}, \ldots, y_{m} \}$ on a metric space $(\Omega,d)$,
\be
\tilde{d}_{\mathrm{LP}}(F,G) &\leq \frac{m}{m+n} \\ 
W_{d}(F,G) &\leq \frac{m}{m+n} \Do.
\ee 
This allows us to extend inequalities bounding the distance between $\pi_{1}$ and the empirical measure on $\{X_{s}^{(1)}\}_{s \leq t}$ for one particular value of $t$ to inequalities  bounding the distance between $\pi_{1}$ and the empirical measure  on $\{X_{s}^{(1)}\}_{s \leq t'}$ for all values of $t' \in [t,T]$ for $T > t$.
\end{remark}

%

Let $F_t$ denote the empirical measure on the elements from the set $\widehat{D}^{(0)}_{0,t}$. In the next lemma, we will show that a bound on the distance between $F_{t}$ and $\pi_{1}$ can be translated into a bound on the distance between the equi-energy process $\{X_{t}^{(0)}\}_{t \in \mathbb{N}}$ and a coupled Markov chain drawn from the limiting chain described in Algorithm \ref{Alg:LEE}.\par
Before stating the result, we give some notation. Fix constants $T_{\mathrm{init}} \leq T_{\mathrm{fin}} \in \mathbb{N}$.  We denote by $\{ Y_{t} \}_{T_{\mathrm{init}} \leq t \leq T_{\mathrm{fin}}}$ a Markov chain that evolves according to Algorithm \ref{Alg:LEE}. Finally, define the constant
\be
\mathcal{C}_{1} = \mathcal{C}_{1}(\epsilon) = \frac{1- p_{ee}}{4}  \frac{p_{ee}}{4}  \left( \frac{cM}{16} - \epsilon \right).
\ee

We then prove our main bound:

\begin{lemma} [Distance Bound] \label{LemEeDistBound}
Fix  $T_{\mathrm{init}} \leq T_{\mathrm{fin}}$ and  $0 < \epsilon < \min \left( \frac{1}{256}, \frac{cM}{192} \right)$, and let $\mathcal{A} = \mathcal{A}(T_{\mathrm{init}}, T_{\mathrm{fin}}, \epsilon)$ be the event that
\be \label{IneqNeededGoodSetForDistBd}
\sup_{T_{\mathrm{init}} \leq t \leq T_{\mathrm{fin}}} \tilde{d}_{\mathrm{LP}}(F_{t}, \pi_{1}) \leq \epsilon.
\ee 
Then for $ T_{\mathrm{init}} \leq S \leq T_{\mathrm{fin}}$ and processes $\{ X_{t}^{(0)} \}_{T_{\mathrm{init}} \leq t \leq T_{\mathrm{fin}}}$, $\{ Y_{t} \}_{T_{\mathrm{init}} \leq t \leq T_{\mathrm{fin}}}$ as above, with any starting points $X_{T_{\mathrm{init}}}^{(0)}, Y_{T_{\mathrm{init}}}$, we have conditional upon $\mathcal{A}$:
\be \label{IneqSimpleEeDistBoundRev}
W_{d}(X_{S}^{(0)}, Y_{S}) \leq  \left(1 -   \mathcal{C}_{1}  \right)^{\lfloor \frac{S - T_{\mathrm{init}}}{4} \rfloor} + 2 \epsilon + (1 - p_{ee})^{ \lfloor \frac{S- T_{\mathrm{init}}}{2} \rfloor } + 4 (S - T_{\mathrm{init}}) e^{-H}.
\ee 
\end{lemma}
\begin{proof}
Throughout this argument, we condition on $\mathcal{A}$ holding.
Briefly, our proof proceeds by first showing that $X_{t}^{(0)}$ quickly enters $V^{-1}(\{ 0 \})$ with high probability, and then showing that after this event it can be coupled to the chain $\{ Y_{t} \}_{T_{\mathrm{init}} \leq t \leq T_{\mathrm{fin}}}$ in such a way that the two chains are rarely far apart. \\
Let $\tau_{1} = \inf \{ t \geq T_{\mathrm{init}} \, : \, V(X_{t}^{(0)})  = 0 \}$ and let $\tau_{2} = \inf \{ t > \tau_{1} \, : \, V(X_{t}^{(0)}) > 0 \}$. We bound the probability that $\tau_{1}$ is large by showing that, regardless of the point $X_{t}^{(0)}$, there is a large probability that $\tau_{1} \leq t + 2$. Writing $U \sim \mathrm{Unif}[0,1]$, we have
 
\be
\P_{0}[V(X_{t+2}^{(0)}) = 0 | X_{t}^{(0)} ] &\geq \P_{0}[V(X_{t+2}^{(0)}) = 0| \inf_{1\leq i \leq M } \vert X_{t+1}^{(0)} - \frac{2i-1}{2M} \vert \leq \frac{c}{8}, X_{t}^{(0)}] \\
&\hspace{2cm}\times \P_{0}[ \inf_{1\leq i \leq M } \vert X_{t+1}^{(0)} - \frac{2i-1}{2M} \vert \leq \frac{c}{8} | X_{t}^{(0)}] \\
&\geq \P_{0}[V(X_{t+2}^{(0)}) = 0| \inf_{1\leq i \leq M } \vert X_{t+1}^{(0)} - \frac{2i-1}{2M} \vert \leq \frac{c}{8}, X_{t}^{(0)}] \\
&\hspace{2cm} \times \frac{p_{ee}}{4} \left( \P[  \inf_{1\leq i \leq M } U - \frac{2i-1}{2M} \vert \leq \frac{c}{16}] - \epsilon \right)  \\
&\geq \P_{0}[V(X_{t+2}^{(0)})= 0| \inf_{1\leq i \leq M } \vert X_{t+1}^{(0)} - \frac{2i-1}{2M} \vert \leq \frac{c}{8}, X_{t}^{(0)}] \\
& \hspace{2cm}\times \frac{p_{ee}}{4} \left( \frac{cM}{16} - \epsilon \right)  \\
&\geq \frac{1- p_{ee}}{4}  \frac{p_{ee}}{4}  \left( \frac{cM}{16} - \epsilon \right)  \\
&= \mathcal{C}_{1}. \label{IneqOneStepToSmallPot}
\ee
The first inequality comes from a simple conditioning, the second inequality comes from inequality \eqref{IneqNeededGoodSetForDistBd}, and the remaining two inequalities are simple computations concerning uniform random variables. \\
Iterating inequality \eqref{IneqOneStepToSmallPot}, we have for all $T_{\mathrm{init}} \leq T \leq T_{\mathrm{fin}}$ that
\be  \label{IneqTau1Bound}
\P_{0}[\tau_{1} > T] \leq \left(1 -   \mathcal{C}_{1}  \right)^{\lfloor \frac{T - T_{\mathrm{init}}}{2} \rfloor}. 
\ee
Next, denote by $\{ Y_{t} ' \}_{T_{\mathrm{init}} \leq t \leq T_{\mathrm{fin}}}$ a second copy of the limiting chain. We couple the chains $\{ Y_{t} \}_{T_{\mathrm{init}} \leq t \leq T_{\mathrm{fin}}}$, $\{ Y_{t} ' \}_{T_{\mathrm{init}} \leq t \leq T_{\mathrm{fin}}}$ by choosing the same value of $p_{t}$ (as given in Algorithm \ref{Alg:LEE}) for both chains at every step $t$, and conditional on the value of $p_{t}$ coupling according to the quantile coupling described in equation \eqref{QuantCoupDef}. We allow for arbitrary starting position $Y_{T_{\mathrm{init}}}'$ of the chain $\{ Y_{t}' \}_{T_{\mathrm{init}} \leq t \leq T_{\mathrm{fin}}}$. It is easy to check by direct computation that
\be 
\E[d(Y_{t+1}, Y_{t+1}') | Y_{t}, Y_{t}', V(Y_{t}) = V(Y_{t}') = 0] \leq (1 - p_{ee}) d(Y_{t}, Y_{t}'). 
\ee 
By inequality \eqref{IneqNeededGoodSetForDistBd}, there exists a Markovian coupling of $\big\{ ( X_{t}^{(0)}, Y_{t}) \big\}_{t}$ (conditional upon $\{ X_{t}^{(1)} \}_{t \in \mathbb{N}}$) so that
\be 
\E_{0}[d(X_{t+1}^{(0)}, Y_{t+1}) | X_{t}^{(0)} = Y_{t} = x, V(x) = 0] \leq 2 p_{ee} \epsilon.
\ee 
Combining these two inequalities, we have by the triangle inequality that there exists a Markovian coupling with the property that
\be \label{IneqSimpleContEstForB}
\E_{0}[d(X_{t+1}^{(0)}, Y_{t+1}) | X_{t}^{(0)}, Y_{t}, V(X_{t}^{(0)}) = V(Y_{t}) = 0] \leq (1 - p_{ee}) d(X_{t}^{(0)}, Y_{t}) + 2 p_{ee} \epsilon.
\ee
For the remainder of this proof, fix $T_{\mathrm{init}} \leq S \leq T_{\mathrm{fin}}$ and let $\mathcal{B} = \mathcal{B}(S)$ be the event that $\tau_{2} > S$ and that $\sup_{T_{\mathrm{init}} \leq t \leq S} V(Y_{t}) = 0$. Iterating the bound in inequality \eqref{IneqSimpleContEstForB}, we have for all $ T_{\mathrm{init}} \leq t \leq S$ that
\be 
\E_{0}[d(X_{t}^{(0)}, Y_{t}) | \tau_{1}, X_{\tau_{1}}^{(0)}, Y_{\tau_{1}}] &= \E_{0}[d(X_{t}^{(0)}, Y_{t}) \textbf{1}_{\mathcal{B}} | \tau_{1}, X_{\tau_{1}}^{(0)}, Y_{\tau_{1}}] + \E_{0}[d(X_{t}^{(0)}, Y_{t}) \textbf{1}_{\mathcal{B}^{c}} | \tau_{1}, X_{\tau_{1}}^{(0)}, Y_{\tau_{1}}] \\
&\leq \E_{0}[ \left( (1 - p_{ee})d(X_{t-1}^{(0)}, Y_{t-1}) + 2 p_{ee} \epsilon \right) \textbf{1}_{\mathcal{B}} | \tau_{1}, X_{\tau_{1}}^{(0)}, Y_{\tau_{1}}] + 4 (S - T_{\mathrm{init}})e^{-H} \\
&\leq \ldots \\
&\leq \frac{2 p_{ee} \epsilon}{p_{ee}} + (1 - p_{ee})^{t - \tau_{1}} + 4 (S - T_{\mathrm{init}})e^{-H}.
\ee 
Combining this with inequality \eqref{IneqTau1Bound}, we have for all $T_{\mathrm{init}} \leq T \leq t \leq T_{\mathrm{fin}}$ that 
\be 
\E_{0}[d(X_{t}^{(0)}, Y_{t})] &\leq \P[ \tau_{1} > T] + 2 \epsilon + (1 - p_{ee})^{t - T} + 4 (t - T_{\mathrm{init}}) e^{-H} \\
&\leq \left(1 -   \mathcal{C}_{1}  \right)^{\lfloor \frac{T - T_{\mathrm{init}}}{2} \rfloor} + 2 \epsilon + (1 - p_{ee})^{t - T} + 4 (t - T_{\mathrm{init}}) e^{-H}.
\ee 
Choosing $T = \lfloor \frac{t + T_{\mathrm{init}}}{2} \rfloor$ gives inequality \eqref{IneqSimpleEeDistBoundRev}, finishing the proof.
\end{proof}

We now relate the above lemma to autocorrelation. Fix $S, T \in \mathbb{N}$ and Let $\{ M_{t} \}_{t=T}^{T+S}$ be any Markov chain with stationary distribution $\pi$, started from stationarity at time $T$ and coupled to $\{ X_{t}^{(0)} \}_{t = T}^{T+S}$. We require that $M_{T}, X_{T}$ be independent but otherwise do not restrict the coupling. With this notation, we have:

\begin{lemma} [Autocorrelation Inequality by Wasserstein Distance] \label{LemmAutoCorrWass} 
Let $S$, $T$ and $\{ M_{t} \}_{t = T}^{T+S}$ be as above. Fix a function $f$ on $[0,1]$ with $\| f \|_{\infty} = 1$, $\| f \|_{\lip} = L$ and $\E_{\pi}[f] = 0$. Fix a starting time $T$ and an ending time $S$, and let $\mathcal{F}$ be the $\sigma$-algebra generated by $\{ X_{t}^{(0)} \}_{t =0}^{T}$. Then:
\be
\left \vert \E_{0}[f(X^{(0)}_{T})f(X^{(0)}_{T+S}) \vert \mathcal{F}] \right \vert \leq  \E_{0} \left [\vert f(X^{(0)}_{T+S}) - f(M_{T+S})\vert \, \Big \vert \mathcal{F} \right].  
\ee
\end{lemma}
\begin{proof}
\be
\left \vert \E_{0}[f(X^{(0)}_{T})f(X^{(0)}_{T+S}) \vert \mathcal{F} ] \right \vert &= \vert \E_{0}[f(X^{(0)}_{T})(f(X^{(0)}_{T+S}) - f(M_{T+S}))\vert \mathcal{F} ] + \E_{0}[f(X^{(0)}_{T})f(M_{T+S})\vert \mathcal{F} ]  \vert \\
&= \left \vert \E_{0}[f(X^{(0)}_{T})(f(X^{(0)}_{T+S}) - f(M_{T+S}))\vert \mathcal{F} ]  \right \vert \\
&\leq \E_{0} \left [\vert f(X^{(0)}_{T+S}) - f(M_{T+S})\vert \, \Big \vert \mathcal{F} \right],
\ee
and the lemma is proved.
\end{proof}

Now we have all the ingredients needed for the proof of Theorem \ref{ThmIntAutoEE}.
\begin{proof}[Proof of Theorem \ref{ThmIntAutoEE}] Combining Lemma \ref{QuantCoupIneq} and Remark \ref{RmkCoupExtraPts}, for any fixed starting time $T >  110592 c^{-2} ( + \epsilon^{-1})^{2} \log \left( 4608 (1 + \epsilon^{-1})^{2} \right) $, we have that
\be 
\P \left[\sup_{T \leq s \leq (1 + \frac{\epsilon}{2})T} W_{d}(\widehat{D}_{0,s}^{(0)}, \pi_{1}) > \epsilon \right] \leq \frac{\epsilon}{12}.
\ee 
By repeatedly applying the same inequality and taking a union bound, for any $\ell \in \mathbb{N}$, we have
\be \label{IneqInitEEBoundFinalIneq}
\P \left[\sup_{T \leq s \leq (1 + \frac{\ell \epsilon}{2})T} W_{d}(\widehat{D}_{0,s}^{(0)}, \pi_{1}) > \epsilon \right] \leq \frac{\epsilon}{12} \sum_{i=0}^{\ell-1} \epsilon^{\frac{1}{2} \ell \epsilon} \leq \frac{\epsilon}{12} \min \Big( \ell , \frac{1}{1 - \epsilon^{\frac{1}{2} \epsilon}} \Big).
\ee 
We now condition on the event that this inequality holds.  Define $A = \frac{1}{2c} \min( p_{ee}, \mathcal{C}_{1}) > 0$. Fix $T_{\mathrm{init}} = T \leq S \leq T_{\mathrm{fin}} = (1 + \frac{k\epsilon}{2})$ and define $T_{\mathrm{init}}' = T_{\mathrm{init}}'(S) = \max \left( T_{\mathrm{init}}, S - \frac{4H}{Ac}  \right)$. Combining inequality \eqref{IneqInitEEBoundFinalIneq} with Lemmas \ref{LemEeDistBound} and \ref{LemmAutoCorrWass}, we have for all $T_{\mathrm{init}} \leq S \leq T_{\mathrm{fin}}$ that there exists a coupling of $\{ Y_{t} \}_{T_{\mathrm{init}}' \leq t \leq T_{\mathrm{fin}}}$ to $\{ X_{t}^{(0)} \}_{T_{\mathrm{init}}' \leq t \leq T_{\mathrm{fin}}}$, with $Y_{T_{\mathrm{init}}'}$ started at stationarity, so that 
\be \label{IneqCoupIntStep}
\E[f(X^{(0)}_{T_{\mathrm{init}}'}) f(X^{(0)}_{S}) | \{ X_{s}^{(1)} \}_{s \in \mathbb{N}}]  &\leq \E_{0}[ | f(X^{(0)}_{S}) - f(Y_{S}) |] \\
&\leq \left( 1 - \mathcal{C}_{1} \right)^{ \lfloor \frac{S - T_{\mathrm{init}}'}{4} \rfloor } + 2 \epsilon + (1 - p_{ee})^{\lfloor \frac{S - T_{\mathrm{init}}'}{2} \rfloor} + 4 (S - T_{\mathrm{init}}') e^{-H}.
\ee 
We now extend this coupling back to time $T_{\mathrm{init}}$ as follows. We draw $Y_{T_{\mathrm{init}}}$ from the stationary measure, and then run $\{ X_{t}^{(0)} \}_{t = T_{\mathrm{init}}}^{T_{\mathrm{init}}'}$, $\{ Y_{t} \}_{t = T_{\mathrm{init}}}^{T_{\mathrm{init}}'}$ independently. We note that under this coupling, $Y_{T_{\mathrm{init}}'}$ is independent of $X_{T_{\mathrm{init}}'}$ and drawn from stationarity, and so we can view this as an extension of the coupling of $\{ Y_{t} \}_{T_{\mathrm{init}}' \leq t \leq T_{\mathrm{fin}}}$ to $\{ X_{t}^{(0)} \}_{T_{\mathrm{init}}' \leq t \leq T_{\mathrm{fin}}}$ used to derive inequality \eqref{IneqCoupIntStep}. Under this coupling, we then have by Lemma \ref{LemmAutoCorrWass} and inequality \eqref{IneqCoupIntStep} the similar bound
\be 
\E[f(X^{(0)}_{T_{\mathrm{init}}}) f(X^{(0)}_{S}) | \{ X_{s}^{(1)} \}_{s \in \mathbb{N}}] & \leq \E_{0}[ | f(X^{(0)}_{S}) - f(Y_{S}) |] \\
&\leq \left( 1 - \mathcal{C}_{1} \right)^{ \lfloor \frac{S - T_{\mathrm{init}}'}{4} \rfloor } + 2 \epsilon + (1 - p_{ee})^{\lfloor \frac{S - T_{\mathrm{init}}'}{2} \rfloor} + 4 (S - T_{\mathrm{init}}') e^{-H} \\
&\leq 2 e^{-A c \lfloor \frac{S - T_{\mathrm{init}}'}{4} \rfloor} + 2 \epsilon +  4 (S - T_{\mathrm{init}}') e^{-H} \\
&\leq 2 \epsilon + 2 \max \left( e^{-A c \lfloor \frac{S - T_{\mathrm{init}}}{4} \rfloor}, e^{-H + 1} \right) + 4 \min \left( (S - T_{\mathrm{init}}) e^{-H}, \frac{4H}{Ac} e^{-H} \right).
\ee 
This completes the proof.
\end{proof}

\subsection{Finite \textit{vs.} Infinite Measures of Autocorrelation} \label{SubsecFinVsInfAut}

We note that Theorem \ref{ThmIntAutoEE} measures mixing properties of the equi-energy sampler by comparing only covariances \newline $\E[f(X_{T})f(X_{T+S})]$ at specific times $T$ and $S$. For quantifying the efficiency of MCMC samplers, it is common to look instead at the following infinite measure of covariance: 
\be \label{EqInfAutoDef}
\mathrm{IC}(T) = \sum_{S > 0} \E[f(X_{T}) f(X_{T+S})].
\ee  
We explain the discrepancy here. For many classes of samplers, the variance of the empirical mean of the sampler based on $U$ steps is approximately $\frac{1}{U} \mathrm{IC}(T)$, for $U$ on all time scales and most choices of $f$, $T$ and starting distribution of $X_{T}$. In particular, for Markov chains on a finite state space with stationary distribution $\pi$, it is well known that $\mathrm{IC}(T)$ is approximately equal to the inverse spectral gap of the Markov chain, which is in turn within a factor of approximately $-\log(\min_{x} \pi(x))$ of the mixing time (see \cite{AlFi94}). Thus, the infinite autocovariance in Equation \eqref{EqInfAutoDef} is closely related to the both the  medium-term and asymptotic variance of estimators derived from the associated sampler, as well as the sampler's mixing properties. Since limiting measures of variance are often more mathematically tractable than the analysis of finite-time mixing properties, and they give similar results, the asymptotic variance is often analyzed. This relationship between finite-time mixing properties and asymptotic variance does not hold for equi-energy samplers, and we argue that the measure of covariance we use in Theorem \ref{ThmIntAutoEE} is normally a more appropriate measure of variation. \par 

We illustrate this point by making exact calculations for a simple equi-energy sampler that mixes very well, but for which the sum \eqref{EqInfAutoDef} diverges. Consider a two energy level equi-energy sampler $\{ X_{t}^{(i)} \}_{t \in \mathbb{N}, i \in \{0,1\}}$. We first fix all of the parameters required in Algorithm \ref{Alg:EE}. The sampler has state space $\Omega = [-\frac{1}{2},\frac{1}{2}]$ and target densities
\be \label{eqn:targleb}
\pi_{0} = \pi_{1} = \lambda,
\ee
the Lebesgue measure on $\Omega$. The proposal density is given by $K_{\mathrm{MH}}(x,A) = \lambda(A)$, and the energy rings are defined by $H(v) \equiv [0,1]$. Finally, we denote the burn-in time by an arbitrary number $T_{b}^{(0)} = T_{b} \in \mathbb{N}$ and assume that $X_{0}^{(1)} \sim \lambda$ and $X_{T_{b}}^{(1)}$ has the mixture of distributions $X_{T_{b}}^{(1)} \sim (1 - p_{ee}) \lambda + p_{ee} \mathrm{Unif}( \{ X_{t}^{(1)} \}_{0 \leq t \leq T_{b}})$. Fix the function $f(x) = x$. 
\begin{lemma} \label{lem:asymvar}
For all $T \geq T_{b}$, we have
\be 
\sum_{S > 0} \E[f(X_{T}^{(0)}) f(X_{T+S}^{(0)})] = \infty.
\ee 
The variance of the unbiased estimator $\widehat{\pi}_{T,T_{b}}$ of $\E_{\pi}[f]$ satisfies 
\be 
\mathrm{Var}[\widehat{\pi}_{T,T_b}(f)] = \frac{1 + 2 p_{ee}^{2}}{6 T} + O \Big( \frac{\log(T)}{T^{2}} \Big).
\ee 
\end{lemma}

\begin{remark}
Thus even though the sum $\sum_{S > 0} \E[f(X_{T_{b}}) f(X_{T_{b}+S})]$ is infinite, the normalized asymptotic variance of the equi-energy sampler $T \mathrm{Var}[\widehat{\pi}_{T,T_b}(f)]$ has a finite limit as $T$ goes to infinity.
\end{remark}
\begin{proof}[Proof of Lemma \ref{lem:asymvar}]
 Let $\mathcal{F}_{t}$ be the $\sigma$-algebra generated by $\{X_{s}^{(1)} \}_{s \leq t}$ and $\{ X_{s}^{(0)} \}_{s  \leq t}$. Recall from Algorithm \ref{Alg:EE} that $p_{t}^{(0)}$, $p_{t}^{(1)}$ are the random variables indicating whether an equi-energy or Metropolis-Hastings step is taken at time $t$, and for $t > T_{b}$ we write $E_{t} = \textbf{1}_{p_{t-1}^{(0)} = \mathrm{EE}}$. We also abuse notation slightly and define $E_{T_{b}} \in \{0,1\}$ conditional on $\{ X_{t}^{(i)} \}_{0 \leq t \leq T_{b}, i \in \{0,1\}}$ by
\be
E_{T_{b}} = \textbf{1}_{X_{T_{b}}^{(0)} \in \{ X_{t}^{(1)} \}_{0 \leq t \leq T_{b}}}.
\ee
We note that $\P[E_{T_{b}} = 1] = p_{ee}$. We then fix $T \geq T_{b}$ and calculate 
\be \label{EqVSimpBiasComp}
\E[X_{T+S}^{(0)} X_{T}^{(0)}] & = p_{ee}^{2} \E[X_{T+S}^{(0)} X_{T}^{(0)} \vert E_{T} = E_{T+S} = 1] \\
&= p_{ee}^{2} \sum_{i=0}^{T} \sum_{j=0}^{ T + S } \frac{1}{(T+1)(T + S+1)} \E[X_{i}^{(1)} X_{j}^{(1)}] \\
&= p_{ee}^{2} \sum_{i=0}^{ T}  \frac{1}{(T+1)(T + S+1)} \E[X_{i}^{(1)} X_{i}^{(1)}] \\
&= \frac{p_{ee}^{2}}{6} \frac{1}{T + S + 1}.
\ee 
In particular, for all $T_{b}$ and all $ T \geq T_{b}$ and $p_{ee} > 0$ we have that the sum $\sum_{S > 0} \E[X_{T} X_{T+S}]$ diverges, albeit slowly. \par 
 To compute the finite-time variance, we note that $\E[X_{t}^{(0)}] = 0$ for all $t$, and write:
\be 
\E\Big[ \big( \sum_{t=T_{b}+1}^{ T_{b}+ T} X_{t}^{(0)}  \big)^{2} \Big] &= \sum_{t=  T_{b}+1}^{T_{b}+ T } \E\Big[ \big( X_{t}^{(0)} \big)^{2}\Big] + 2 \sum_{ T_{b}+ 1 \leq s < t \leq T_{b}+T} \E[X_{s}^{(0)} X_{t}^{(0)}] \\
&= \frac{T}{6} + \frac{ p_{ee}^{2}}{3} \sum_{T_{b}+1 \leq s < t \leq T_{b}+ T} \frac{1}{t+1} \\
&= \frac{T}{6} + \frac{ p_{ee}^{2}}{3} \sum_{ t=T_{b}+1}^{T_{b}+ T} \Big(1 - \frac{2}{t+1} \Big) \\
&= \frac{T}{6} \Big( 1 + 2 p_{ee}^{2} - 2 \frac{\log(1 + \frac{T}{T_{b}})}{T} \Big) + O(1),
\ee 
and the proof is finished.
\end{proof}
 Lemma \ref{lem:asymvar} raises two main points. First, in studying equi-energy samplers, we cannot use the infinite measure of autocovariance \eqref{EqInfAutoDef} as a useful proxy for mixing properties, as one can when studying Markov chains. Second, from a practical point of view, this is not a large concern. The autocovariance is principally of interest as a way to bound the asymptotic variances of estimators of the same form as $\widehat{\pi}_{T, T_b}(f)$, and we have seen that this is often quite small. \par

\subsection{Convergence of Equi-Energy Samplers} \label{SubSecAbsConvThm} 
We now use similar techniques to those in subsection \ref{SubsecFirstEeProof} to show the convergence of equi-energy samplers in a more general setting. We begin with some new notation. Define the distance $D$ between kernels $K_{1}, K_{2}$ on a common metric space $(\Omega, d)$ by
\be \label{EqDistOnKernDef} 
D(K_{1}, K_{2}) = \sup_{x \in \Omega} W_{d}(K_{1}(x,\cdot), K_{2}(x,\cdot)).
\ee

We now provide additional notation for the limiting sampler defined in Algorithm \ref{Alg:LEE}. We define for $i < \mathcal{K}$ the transition kernel of the limiting chain to be
\be
K_{\infty}^{(i)} = (1-p_{ee})K_{\mathrm{MH}}^{(i)} + p_{ee} K_{\mathrm{EE}}^{(i)},
\ee
where $K_{\mathrm{MH}}^{(i)}$ is the Metropolis-Hastings chain associated with proposal kernel $K_{\mathrm{MH}}$ and target $\pi_{i}$, and the limiting equi-energy kernel $K_{\mathrm{EE}}^{(i)}$ is the Metropolis-Hastings chain associated with proposal kernel
\be
p_{\mathrm{EE}}^{(i)}(x,A) = \frac{\pi_{i+1}\big(A \cap V^{-1}(H(V(x) )) \big) }{\pi_{i+1}\big(V^{-1}(H(V(x)))\big)}
\ee
and target distribution $\pi_{i}$. For $i = \mathcal{K}$, define $K_{\infty}^{(\mathcal{K})}$ to be the Metropolis-Hastings chain with proposal distribution $K_{\mathrm{MH}}$ and target $\pi_{\mathcal{K}}$. Recall that the burn-in time for the $i^{\mathrm{th}}$ chain is  $T_{b}^{(i)}$; that is, we wait until time $T_{b}^{(i)}$ before starting chain $X_{t}^{(i)}$.  \par
We denote by $\kappa_{\mathrm{MH}}^{(i)}$ and $\kappa_{\mathrm{EE}}^{(i)}$ the curvature of the Metropolis-Hastings and limiting equi-energy chains at level $i$ respectively. The curvature of $\left(K_{\infty}^{(i)} \right)^{q}$ is denoted by $\kappa_{\infty}^{(i),\, q}$ for $q \in \mathbb{N}$. 
\begin{remark}
It is easy to check that, for all $p, q \in \mathbb{N}$,
\be 
\kappa_{\infty}^{(i), \,pq} \geq 1 - \left( 1 - \kappa_{\infty}^{(i),\,p} \right)^{q}.
\ee 
In general this inequality is far from tight. Most significantly for our purposes, there are many natural examples for whcih $\kappa_{\infty}^{(i),1} < 0$ while $\kappa_{\infty}^{(i),2}$ or $\kappa_{\infty}^{(i),3}$ are positive.
\end{remark}
Let $D_\Omega$ denote the diameter of $(\Omega, d)$.
\begin{assumption} \label{ass:1} The following assumptions will be used for our main result in this section. 
\begin{enumerate}
\item $\Omega$ is a  compact subset of  $\mathbb{R}^{m}$ with $\Do \leq 1$ and $d$ is the Euclidean metric. \label{AbsThmAsm1} 
\item For $v \in \mathbb{R}$, the set $V^{-1}(H(v))$ is a union of at most $N_{1}$ balls or rectangles in $\mathbb{R}^{m}$. \label{AbsThmAsm5} 
\item  $\pi_{i}$ have densities uniformly bounded above by $N_{2} > 0$, for $0 \leq i \leq \mathcal{K}$. \label{AbsThmAsm3}
\item For some fixed $\alpha, k > 0$ and all $0 \leq i \leq \mathcal{K}$, we have $\kappa_{\infty}^{(i),k} > \alpha$. \label{AbsThmAsm4} 
\item There exists $\mathcal{C} > 0$ so that $W_{d}(K_{\infty}^{(i)}(x,\cdot), K_{\infty}^{(i)}(y,\cdot)) \leq \mathcal{C} d(x,y)$ for $0 \leq i \leq \mathcal{K}$. \label{AbsThmAsm6}
\item The sets $V^{-1}(H(V(x)))$ have probability uniformly bounded below, \textit{i.e.}, $\pi_{i}(H(V(x))) \geq N_{3}$ for $0 \leq i \leq \mathcal{K}$. \label{AbsThmAsm7}
\end{enumerate}
\end{assumption}
We now briefly discuss the assumptions above.

\begin{enumerate}
\item By rescaling $d$, the $D_{\Omega} \leq 1$ from Assumption \ref{AbsThmAsm1} can be replaced by $D_{\Omega} < \infty$. A finite diameter is largely used for the covering argument in equation \eqref{IneqDefOfQ}; we believe that, with appropriate modifications, our arguments should also hold for non-compact state spaces. The assumption of a Euclidean metric is useful only in that it allows us to simply state a regularity condition (see the immediately following point), and is not used in any fundamental way.
\item Assumption \ref{AbsThmAsm5} is much stronger than necessary, but is easy to state and verify. This assumption is used in Lemma \ref{LemWassAppKern} to go from a bound on the distance between two measures to a bound on the distance between their restrictions to small sets. Defining the $\delta$-thickening as in equation \eqref{EqThickDef}, this assumption can be relaxed to the milder regularity assumption that $\sup_{i,x} \pi_{i}(V^{-1}(H(V(x)))_{\delta}) = O(\delta^{a})$ for some $a > 0$ without changing our conclusions. Without any regularity assumptions, the type of bound obtained in Lemma \ref{LemWassAppKern} is generally false.
\item This assumption holds for most reasonable examples and is also used in Lemma \ref{LemWassAppKern}.
\item  Assumption \ref{AbsThmAsm4} is absolutely necessary for our proof. All of our bounds depend on $\alpha$ and many become useless as $\alpha$ goes to 0, just as in the theory of Markov chains with positive curvature. As in \cite{Olli09}, many bounds hold for chains that have $\kappa(x,y) > \alpha > 0$ only for $d(x,y) \geq \epsilon > 0$ sufficiently large (see e.g. Theorem \ref{ThmIntAutoEE}).
\item In most examples of interest, we have $\mathcal{C} \geq 1$ in Assumption \ref{AbsThmAsm6}. Assumption \ref{AbsThmAsm6} allows us to control the distance between powers of two kernels in terms of the distance between the underlying kernels. Without any continuity assumption, it is easy to find sequences of kernels $\{Q_{n},K_{n} \}_{n \in \mathbb{N}}$ so that $ \lim_{n \rightarrow \infty} D(Q_{n},K_{n})= 0$ but $\lim_{n \rightarrow \infty} D(Q_{n}^{2}, K_{n}^{2}) = 1$. Our Lipschitz condition could be replaced by a uniform continuity condition on the maps $x\mapsto K_\infty^{(i)}(x,\cdot)$.
\item Assumption \ref{AbsThmAsm7} holds for any reasonable function $H$, and in particular always holds for the energy rings used by \cite{KZW06}. It is used in Lemma \ref{LemWassAppKern}.
\end{enumerate}

Now we are ready to state the main result of this section which gives the finite sample bounds for the EE algorithm and the convergence of the associated kernels to the limiting kernel. 
\begin{thm} \label{ThmNonQuantConv}
Let Assumptions \ref{ass:1} hold and fix $\epsilon > 0$. Then, for $\eta \in \mathbb{R}$, there exist a sequence of times $T_{b}^{(i)} = T_{b}^{(i)}(\eta), t = t(\eta) \leq \eta$ so that, for any (deterministic or random) sequence of starting points $X_{T_{b}^{(i)}}$,
we have
\be \label{ThmNonQuantConvPt1}
\lim_{\eta \rightarrow \infty} \P(D(K^{(0)}_t, K^{(0)}_\infty) > \epsilon) = 0.
\ee
Under the same assumptions and notation, for all (deterministic or random) sequences of starting points $X_{T_{b}^{(i)}}$, $k$ as in item \ref{AbsThmAsm4}  of Assumption \ref{ass:1},  $0 < r < \frac{1}{5(k+1)}$, 1-Lipschitz functions $f$ and $s \in \mathbb{N}$,
\be \label{ThmNonQuantConvPt2}
\P[\vert \widehat{\pi}_{s,t}(f) - \pi(f) \vert > r] \leq 2 k e^{- \alpha r^{2} \lfloor \frac{ s }{2k} \rfloor} + a_t,
\ee 
where $a_t$ is independent of $s$ and satisfies $\lim_{\eta \rightarrow \infty} a_{t(\eta)} = 0$.
\end{thm}

\begin{remark}
Our proof also gives analogous bounds for large $r$; we omit them because they are less interesting, especially in our setting of 1-Lipschitz functions on spaces with diameter 1. \par 
We mention that we can choose $ \eta - 2 \leq t(\eta) \leq \eta$ and the result will still hold. The result is not true for arbitrary choices of $\{ T_{b}^{(i)} \}_{i=0}^{\mathcal{K}}$.
\end{remark}

\subsection{Proof of Theorem \ref{ThmNonQuantConv}}

\subsubsection{Proof sketch}
Denote by $F_{t}^{(i+1)}$ the empirical measure associated with $\{X_{s}^{(i+1)}\}_{s = T_{b}^{(i+1)}}^{t}$. Also, for $0 \leq i \leq \mathcal{K}$, we recall from Section \ref{SubsecNotCond} that
\be
\P_{i}[\cdot] \equiv \P \Big[ \cdot \vert \{ X_{t}^{(j)} \}_{t \in \mathbb{N}, j \geq i+1} \Big] 
\ee
is the probability of an event conditioned on the entire evolution of the equi-energy process at all levels $j \geq i+1$. We note that, conditioned on $ \{ X_{t}^{(j)} \}_{t \in \mathbb{N}, j \geq i+1}$, the sequence $\{ X_{t}^{(i)} \}_{t \in \mathbb{N}}$ is a (time-inhomogenous) Markov chain. That is,
\be 
\P_{i}[X_{t+1}^{(i)} \in \cdot \vert X_{0}^{(i)}, \ldots, X_{t}^{(i)}] = \P_{i}[X_{t+1}^{(i)} \in \cdot \vert X_{t}^{(i)}].
\ee 
This conditional-Markovianity will allow us to use some results from the theory of Markov chains, and in particular our Theorem \ref{ThmConcAlmostPositiveAlt}. The proof will proceed via an induction argument on $i$, entailing the following steps. 

\begin{itemize}
\item We begin by using the above conditioning argument and Theorem \ref{ThmConcAlmostPositiveAlt}  to show that, if 
some power of $K_{\infty}^{(i+1)}$ has positive curvature and $D(K_{t}^{(i+1)}(x,\cdot), K_{\infty}^{(i+1)}(x,\cdot)) \ll 1$ for all $T \leq t \leq T+S$, then any particular 1-Lipschitz function $f$ satisfies
\be \label{eqn:pskapp}
S^{-1} \sum_{t=T+1}^{T+S} f(X_{t}^{({i+1})}) \approx \pi_{i+1}(f)
\ee 
with high probability. This is the first half of the proof of Lemma \ref{LemmIndErr1}.
\item By a covering argument on the space of bounded Lipschitz functions, we show in Lemma \ref{LemmIndErr1} that \eqref{eqn:pskapp} implies 
\be
W_{d}(\pi_{i+1}, F_{t}^{(i+1)}) \ll 1.
\ee
\item Finally, we show in Lemma \ref{LemWassAppKern}  that
$W_{d}(F_{t}^{(i+1)}, \pi_{i+1}) \ll 1$ implies that \be 
\sup_{t \leq s \leq u} D(K_{s}^{(i)}, K_{\infty}^{(i)}) \ll 1
\ee
for some $u$ that grows with $t$. This returns us to the situation in the first line of the proof sketch, with $i+1$ replaced by $i$.
\end{itemize}

We now discuss the difference between our arguments and the argument in \cite{KZW06} in light of the subsequent modification by Atchad{\'e} and Liu. Both our argument and the argument in \cite{KZW06} rely on a main inductive step, showing that if all kernels $\{ K_{t}^{(i+1)} \}_{t=T}^{T+S}$ are `close' to a limiting kernel, then the set $\{ X_{t}^{(i+1)} \}_{t = T}^{T+S}$ will be `close' to a sequence of draws from $\pi_{i+1}$, which in turn will force $K_{T+S}^{(i)}$ to be close to its limiting kernel. The argument in \cite{KZW06} describes a sequence of kernels $K_{n}$ as converging to $K_{\infty}$ if $\lim_{n \rightarrow \infty} K_{n}(x,A) = K_{\infty}(x,A)$ for all $x \in \Omega, A \subset \Omega$, and they describe closeness of random variables in terms of total variation. This is a very weak type of convergence for kernels, and a very strong form of convergence for the associated stochastic process. The known gap in the convergence argument of \cite{KZW06} (as pointed out by Atchad{\'e} and Liu) is made in going from the former form of convergence to the latter. In our argument, we say that a sequence of kernels $K_{n}$ converges to a kernel $K$ if $\lim_{n \rightarrow \infty} \sup_{x} W_{d}(K_{n}(x,\cdot), K(x,\cdot)) = 0$, and we describe closeness of random variables in terms of their Wasserstein distance. This is a strengthening of the metric on kernels and a weakening of the metric on random variables relative to \cite{KZW06}. This weakening lets us prove much more uniform convergence bounds and thus avoid the above difficulty. 
\subsection{Preliminary results} \label{SubsecPrelRes}

We begin with some simple general results related to gluing together and taking apart couplings. Fix a set $B \subset \Omega$, and for $\delta > 0$, define the $\delta$-thickening of $B$ by:
\be \label{EqThickDef}
B_{\delta} = \{ x \, : \, \inf_{y \in B} d(x,y) < \delta \}.
\ee  
\begin{lemma} [Coupling Faraway Points] \label{LemmaCoupPastBound}
Fix $\epsilon > 0$ and consider two measures $\mu$, $\nu$ on a metric measure space $(\Omega, d)$ with the property that $W_{d}(\mu,\nu) < \epsilon$. Then it is possible to couple random variables $X \sim \mu, Y \sim \nu$ so that, for any $\delta >0$,
\be 
\P[X \in B, Y \notin B_{\delta}]  \leq \frac{\epsilon}{\delta}.
\ee 
\end{lemma}
\begin{proof}
Since $W_{d}(\mu, \nu) < \epsilon$ is a strict inequality, it is possible to couple $X \sim \mu, Y \sim \nu$ so that $\E[d(X,Y)] < \epsilon$ as well. We note that
\be 
\epsilon &\geq \E[d(X,Y)] \\
&\geq \delta \P[d(X,Y) > \delta] \\
&\geq \delta\, \P[X \in B, Y \notin B_{\delta}],
\ee 
finishing the proof.
\end{proof}
This allows us to prove the following result, whose proof is given in the Appendix.
\begin{lemma} [Subcoupling Construction] \label{CorCoupPastCon}
Consider two measures $\mu$, $\nu$ on a metric measure space $(\Omega, d)$ with the property that $W_{d}(\mu,\nu) < \epsilon$. Then for any set $B$ and any $\delta > 0$, 
\be 
W_{d}( \mu \vert_{B}, \nu \vert_{B}) \leq \mathrm{diam}(\Omega) \frac{\epsilon}{\delta} + \mathrm{diam}(\Omega) \nu(B_{\delta} \backslash B) + \frac{\epsilon}{\mu(B)}.
\ee 
\end{lemma}

For a set $B$ on a metric space $(\Omega, d)$ and $\epsilon > 0$, define 
\be
B_{-\epsilon} = ((B^c)_{\epsilon})^c,
\ee
so that
\be
B_{-\epsilon} \subset B \subset B_{\epsilon}.
\ee

\begin{lemma} \label{LemTwoMeasComp}
Fix $\epsilon > 0$ and a metric space $(\Omega,d)$ and let $\mu, \nu$ be two distributions with $W_{d}(\mu, \nu) < \epsilon$. Then, for any $\delta > 0$ and any measurable set $B$, we have 
\be 
\nu(B) \geq \mu(B_{-\delta}) - \frac{\epsilon}{\delta}.
\ee 
\end{lemma}
\begin{proof}
Fix $\gamma > 0$ and let $X \sim \mu$ and $Y \sim \nu$ be coupled so as to satisfy $\E[d(X,Y)] \leq \gamma + W_{d}(\mu, \nu)$. We have
\be 
\nu(B) &= \P[Y \in B] \\
&\geq \P[X \in B_{-\delta}, d(X,Y) < \delta] \\
&\geq \P[X \in B_{-\delta}] - \P[d(X,Y) > \delta] \\
&\geq \mu(B_{-\delta}) - \frac{\epsilon + \gamma}{\delta}.
\ee 
Letting $\gamma$ go to 0 proves the claim.
\end{proof}

Recall that $F_{t}^{(i+1)}$ is the empirical measure associated with the set $\{ X_{s}^{(i+1)} \}_{s = T_{b}^{(i+1)}}^{t}$. Also, for all $\gamma > 0$ and $x \in \Omega$, define
\be 
\mathcal{H}_{\gamma} = V^{-1}\big(H(V(x))\big)_{\gamma} \backslash V^{-1}\big(H(V(x))\big).
\ee 
We have:

\begin{lemma} [Wasserstein Approximation of Kernels] \label{LemWassAppKern}
Let Assumptions \ref{ass:1} hold. Fix $t \in \mathbb{N}$ and assume 
\be \label{NonQMainLemmaMainAssump}
W_{d}(F_{t}^{(i+1)}, \pi_{i+1}) < \epsilon.
\ee 
Then 
\be 
D(K_{t}^{(i)}, K_{\infty}^{(i)}) = O(p_{ee} \sqrt{\epsilon})
\ee 
\end{lemma}

\begin{proof}
Fix $\gamma > 0$. We compute
\be \label{IneqUseCoupSubCor}
D(K_{\infty}^{(i)}, K_{t}^{(i)}) &= \sup_{x \in \Omega} W_{d}(K_{\infty}^{(i)}(x,\cdot), K_{t}^{(i)}(x,\cdot)) \\
&= p_{ee} \sup_{x \in \Omega} W_{d}(K_{EE,\infty}^{(i)}(x,\cdot), K_{EE,t}^{(i)}(x,\cdot)) \\
 &\leq  p_{ee}\frac{\epsilon}{\gamma}  + p_{ee} \pi_{i+1}(\he) 
+ p_{ee}\frac{\gamma}{F_{t}^{(i+1)}(H(V(x)))} \\
&\leq p_{ee} \frac{\epsilon}{\gamma} + p_{ee} \pi_{i+1}(\he) 
+p_{ee} \frac{\gamma}{\pi_{i+1}(H(V(x))_{-\gamma}) - p_{ee} \frac{\epsilon}{\gamma} },
\ee 
where the first inequality follows from Lemma \ref{CorCoupPastCon} (with bounds supplied by inequality \eqref{NonQMainLemmaMainAssump} and the fact that  $\Do \leq 1$) and the second inequality is from Lemma \ref{LemTwoMeasComp}. Choosing $\gamma = \sqrt{\epsilon}$ and recalling items 2,3 and 6 from Assumption \ref{ass:1}, we have
\be  
\frac{1}{p_{ee}}D(K_{\infty}^{(i)}, K_{t}^{(i)}) &\leq \frac{\epsilon}{\gamma} + \pi_{i+1}(\he) 
+ \frac{\gamma}{\pi_{i+1}(H(V(x))_{-\gamma}) - \frac{\epsilon}{\gamma} } \\
& \leq  \Big( \sqrt{\epsilon} + \omega(m) N_{1}N_{2} \sqrt{\epsilon} + \frac{\sqrt{\epsilon}}{N_{3} - \omega(m)  N_{1} N_{2}\sqrt{\epsilon} - \sqrt{\epsilon}} \Big),
\ee
where $\omega(m)$ is a constant that depends only on $m$ describing the rate at which the volume of a $\delta$-thickening of a rectangle grows. The claim follows.
\end{proof}

We need to define one further constant, $\mathcal{N}$. By Theorem 2.7.1 of \cite{van1996weak}, for any $\Omega \subset \mathbb{R}^{m}$ of diameter 1, for all $\gamma > 0$, there exist $N(\gamma) \in \mathbb{N}$ and constant $\mathcal{N}$ with 
\be \label{IneqDefOfQ}
\log N(\gamma)\leq \mathcal{N} \gamma^{-m}
\ee 
and  functions $\{ \phi_{j} \}_{j=1}^{N(\gamma)}$ such that for any 1-Lipschitz function $f$ on $\Omega$ there exists some $1 \leq j \leq N(\gamma)$ with
\be\label{eqn:covbd}
\|f - \phi_j\|_\infty \leq \gamma.
\ee
Next, for $0 < \epsilon, \delta < 1$ and $S \in \mathbb{N}$, define the functions
\be \label{EqFirstHFuncs}
\mathcal{H}_{1}(\epsilon, \delta, S) &= \max \Big( \frac{8 \left(2 + k\mathcal{C}^{k} \right)}{\alpha \epsilon}, \frac{4}{\epsilon} S, \frac{32 k^{3}}{\alpha \epsilon^{2}} \big( 2 + \log(2k) + \log \left( \delta^{-1} \right) + \mathcal{N} \big( \frac{4}{\epsilon} \big)^{m} \big) \Big), \\ 
\mathcal{H}_{2}(\epsilon, \delta, S) &=  \frac{16 \left(2 + k \mathcal{C}^{k} \right)}{\epsilon \alpha}. 
\ee

We then have the following inductive argument:
\begin{lemma} [Inductive Error Bound 1] \label{LemmIndErr1}

Fix $i \in \mathbb{N}$, let Assumptions \ref{ass:1} hold and fix constants $0 < \delta_{i}, \epsilon_{i}, \epsilon_{i+1}$ and $ S_{i}, S_{i+1}$ that satisfy the inequalities
\be 
S_{i+1} &> \mathcal{H}_{1}(\epsilon_{i}, \delta_{i}, S_{i}) \\
\epsilon_{i+1}^{-1} &> \mathcal{H}_{2}(\epsilon_{i}, \delta_{i}, S_{i}).
\ee 
Fix also any burn-in time $T_{b}^{(i)} \in \mathbb{N}$. Denote by $\mathcal{A} = \mathcal{A}(\epsilon_{i+1}, T_{b}^{(i)}, S_{i+1}, S_{i})$ the event that 
\be \label{IneqInd1Hyp1}
\sup_{T_{b}^{(i)} \leq t \leq T_{b}^{(i)} + S_{i+1} + S_{i}} D(K_{t}^{(i+1)}, K_{\infty}^{(i+1)}) \leq \epsilon_{i+1}.
\ee 

Then for any starting point $X_{T_{b}^{(i)}}$ of the chain at level $i$, we have
\be 
\P_{i}\big[\sup_{T_{b}^{(i)} + S_{i+1} \leq t \leq T_{b}^{(i)} + S_{i+1} + S_{i}} W_{d}(F_{t}^{(i)}, \pi_{i}) < \epsilon_{i} \vert \mathcal{A} \big] \geq 1 - \delta_{i}.
\ee  
\end{lemma}

\begin{proof}
We begin by noting that the event $\mathcal{A}$ is measurable in the $\sigma$-algebra $\Sigma$ generated by $\{X_{t}^{(i+1)}\}_{t \leq T_{b}^{(i)} + S_{i+1} + S_{i}}$. In addition, conditional upon $\Sigma$, we have that $\{ X_{t}^{(i)} \}_{T_{b}^{(i)} \leq t \leq T_{b}^{(i)} + S_{i+1} + S_{i}}$ is a time-inhomogenous Markov chain. This means that we can apply Theorem \ref{ThmConcAlmostPositiveAlt} conditional upon $\Sigma$ and then the event $\mathcal{A}$, despite the fact that unconditionally  $\{ X_{t}^{(i)} \}_{T_{b}^{(i)} \leq t \leq T_{b}^{(i)} + S_{i+1} + S_{i}}$ is not a Markov process (and thus Theorem \ref{ThmConcAlmostPositiveAlt} would not apply to it directly). \par 
Fix an integer $0 \leq c \leq k-1$. We apply Theorem \ref{ThmConcAlmostPositiveAlt}, with target kernel $K = \big( K_{\infty}^{(i+1)} \big)^{k}$ and approximating kernels $\{ K_{s} \} = \Big\{ \big( K_{T_{i+1} + c + ks}^{(i+1)} \big)^{k} \Big\}$. In the notation of Theorem \ref{ThmConcAlmostPositiveAlt}, we have:
\begin{itemize} 
\item $\Sigma_{T_b, T, \infty} \leq D_{\Omega} \leq 1$ and we can choose $\mathcal{V} \equiv 1$ (and thus $\mathcal{C}_{\mathcal{V}} = 0$) by item \eqref{AbsThmAsm1} of Assumption \ref{ass:1}. 
\item By inequality \eqref{IneqInd1Hyp1}, we have $W_{d}(K_{s}^{(i+1)}, K_{\infty}^{(i+1)}) \leq \epsilon_{i+1}$. By Lemma \ref{LemmContPow} and item \eqref{AbsThmAsm6} of Assumption \ref{ass:1}, this implies $\sup_{x \in \Omega} W_{d}(K_{s}(x,\cdot), K(x,\cdot)) \leq 2 \epsilon_{i+1} \left( 2 + \mathcal{C}^{k+1} \right)$. 
\end{itemize}

Plugging these estimates into Theorem  \ref{ThmConcAlmostPositiveAlt}, we find that for all $r < \frac{1}{5}, \epsilon_{i+1} \ll 1$ sufficiently small and all $1$-Lipschitz functions $f$, 
\be 
\P_{i}\Big[ \Big \vert \frac{1}{\lfloor \frac{S_{i+1} - c}{k} \rfloor} \sum_{s=0}^{ \lfloor \frac{S_{i+1} - c}{k} \rfloor} f(X_{T_{b}^{(i)} + c + ks}^{(i)}) - \E_{i} \Big(  \frac{1}{\lfloor \frac{S_{i+1} - c}{k} \rfloor} \sum_{s=0}^{ \lfloor \frac{S_{i+1} - c}{k} \rfloor } f(X_{T_{b}^{(i)} + c + ks}^{(i)}) \Big) \Big \vert \geq r  \vert \mathcal{A} \Big] \leq 2 e^{-\alpha r^{2} \lfloor \frac{ S_{i+1} - k }{k} \rfloor}.
\ee
Taking a union bound over the terms $0 \leq c \leq k-1$,
\be \label{IneqInd1ConcBound1}
\P_{i}\Big[ \Big \vert \frac{1}{S_{i+1} +1} \sum_{s=0}^{ S_{i+1} } f(X_{T_{b}^{(i)} + s}^{(i)}) - \E_{i} \big( \frac{1}{S_{i+1} + 1} \sum_{s=0}^{ S_{i+1} } f(X_{T_{b}^{(i)} + s}^{(i)}) \big) \Big \vert \geq k r \vert \mathcal{A}\Big] \leq 2k e^{-\alpha r^{2} \lfloor \frac{ S_{i+1} - k }{k} \rfloor}.
\ee
Define the quantity
\be \label{EqDefScriptD}
\mathcal{D} = \Big ( \frac{2 \epsilon_{i+1}}{\alpha}  + \frac{1}{S_{i+1} \alpha} \Big) \Big(2 + k \mathcal{C}^{k} \Big).
\ee 
By Lemma \ref{LemmContPow} and item \ref{AbsThmAsm6} of Assumption \ref{ass:1},
\be \label{IneqAppCorrOnce1}
\sup_{x \in \Omega} W_{d} \Big(\left( K_{t}^{(i+1)} \right)^{k}(x,\cdot), \left( K_{\infty}^{(i+1)} \right)^{k}(x,\cdot) \Big) \leq  \epsilon_{i+1} \left(2 + k \mathcal{C}^{k} \right).
\ee 
By item \eqref{AbsThmAsm1} of Assumption \ref{ass:1}, $\sup_{x \in \Omega} E(x) \leq 1$. Applying Corollary \ref{CorContPow} with this bound on the eccentricity and with inequality \eqref{IneqAppCorrOnce1}, we have  
\be 
\Big| \E_{i} \Big( \frac{1}{S_{i+1} + 1} \sum_{s=0}^{ S_{i+1} } f(X_{T_{b}^{(i)} + s}^{(i)}) \Big) - \pi_{i}(f) \Big| \leq \mathcal{D} .
\ee

Combining this with inequality \eqref{IneqInd1ConcBound1}, we find
\be 
\P_{i}\Big[\Big \vert \frac{1}{S_{i+1}+1} \sum_{s=0}^{ S_{i+1} } f(X_{T_{b}^{(i)} + s}^{(i)}) - \pi_{i}(f) \Big \vert \geq k r +  \mathcal{D} \vert \mathcal{A} \Big] \leq 2k e^{-\alpha r^{2} \lfloor \frac{ S_{i+1} }{k} \rfloor}.
\label{IneqInd1ConcBound1Alt}
\ee
Fix $\gamma > 0$. Following the notation set up in Equation \eqref{eqn:covbd}, we recall that for all measures $\mu, \nu$ and all 1-Lipschitz functions $f$,
\be 
\vert \mu(f) - \nu(f) \vert \leq \gamma + \sup_{1 \leq j \leq \mathcal{N} \gamma^{-m}} \vert \mu(\phi_{j}) - \nu(\phi_{j}) \vert. 
\ee 
By the Kantorovitch-Rubinstein duality theorem, we have for all $s \in \mathbb{N}$
\be \label{IneqUsingKantRubAbs}
W_{d}(F_{s}^{(i)}, \pi_{i}) &\leq \sup_{\|f \|_{\lip} = 1} \vert F_{s}^{(i)}(f) -  \pi_{i}(f) \vert \\
&\leq \gamma + \sup_{1 \leq j \leq \mathcal{N} \gamma^{-m}} \vert F_{s}^{(i)}(\phi_{j}) - \pi_{i}(\phi_{j}) \vert. 
\ee 
Combining this with inequality \eqref{IneqInd1ConcBound1Alt}, we have for $r < \frac{1}{5}$ that 
\be \label{IneqToBeUsedMuchLater}
\P_{i}\Big[ W_{d}(F_{T_{b}^{(i)} + S_{i+1}}^{(i)}, \pi_{i}) \geq \gamma + k r +   \mathcal{D}  \vert \mathcal{A} \Big] \leq 2k e^{\mathcal{N} \gamma^{-m} } e^{-\alpha r^{2} \lfloor \frac{ S_{i+1} - k }{k} \rfloor}.
\ee

By Remark \ref{RmkCoupExtraPts}, this implies that for all $t \geq 0$, 

\be 
\P_{i} \left[ \sup_{T_{b}^{(i)} + S_{i+1} \leq s \leq T_{b}^{(i)} + S_{i+1} + t} W_{d}(F_{t}^{(i)}, \pi_{i}) \geq \frac{t}{t + S_{i+1}} + \gamma + k r + \mathcal{D} \vert \mathcal{A} \right] \leq 2k e^{\mathcal{N} \gamma^{-m} } e^{-\alpha r^{2} \lfloor \frac{ S_{i+1} - k }{k} \rfloor}.
\ee
Thus, for
\be 
S_{i+1} &> \mathcal{H}_{1}(\epsilon_{i}, \delta_{i}, S_{i}),\,
\epsilon_{i+1}^{-1} > \mathcal{H}_{2}(\epsilon_{i}, \delta_{i}, S_{i}),
\gamma =  \frac{\epsilon_{i}}{4},\, 
r = \frac{\epsilon_{i}}{4k},
\ee 
we have
\be 
\P_{i}\Big[\sup_{T_{i+1} + S_{i+1} \leq s \leq T_{i+1} + S_{i+1} + S_{i}} W_{d}(F_{t}^{(i)}, \pi_{i}) > \epsilon_{i} \vert \mathcal{A} \Big] \leq \delta_{i}.
\ee
This completes the proof.
\end{proof}

 By Lemma \ref{LemWassAppKern}, there exists some constant $\mathcal{B}$ so that, for any $t \in \mathbb{N}$ and any $\epsilon > 0$ sufficiently small, the inequality $W_{d}(F_{t}^{(i+1)}, \pi_{i+1}) < \epsilon$ implies the inequality 
\be \label{IneqIndHypFinalSimpleBound}
D(K_{\infty}^{(i)},K_{t}^{(i)}) \leq \mathcal{B} p_{ee} \sqrt{\epsilon}.
\ee

Recalling the functions $\mathcal{H}_{1}$ and $\mathcal{H}_{2}$ given in Equation \eqref{EqFirstHFuncs}, we use this constant $\mathcal{B}$ to define
\be  \label{eqn:consh3h4}
\mathcal{H}_{3}(\epsilon, \delta, S) &= \mathcal{H}_{1} \big( \frac{\epsilon^{2}}{\mathcal{B}^{2} p_{ee}^{2}}, \delta, S \big), \, \quad
\mathcal{H}_{4}(\epsilon, \delta, S) = \mathcal{H}_{2} \big( \frac{\epsilon^{2}}{\mathcal{B}^{2} p_{ee}^{2}}, \delta, S \big).
\ee 
This allows us to prove the following stronger inductive claim:

\begin{cor} [Inductive Error Bound 2]  \label{LemmIndErrFinal}
Fix $i \in \mathbb{N}$, let Assumptions \ref{ass:1} hold and fix constants $0 < \delta_{i}, \epsilon_{i}, \epsilon_{i+1}$ and $ S_{i}, S_{i+1}$ that satisfy the inequalities
\be 
S_{i+1} &> \mathcal{H}_{3}(\epsilon_{i}, \delta_{i}, S_{i}) \\
\epsilon_{i+1}^{-1} &> \mathcal{H}_{4}(\epsilon_{i}, \delta_{i}, S_{i}).
\ee 
Fix also any burn-in time $T_{b}^{(i)} \in \mathbb{N}$. Denote by $\mathcal{A} = \mathcal{A}(\epsilon_{i+1}, T_{b}^{(i)}, S_{i+1}, S_{i})$ the event that 
\be \label{IneqInd1Hyp}
\sup_{T_{b}^{(i)} \leq t \leq T_{b}^{(i)} + S_{i+1} + S_{i}} D(K_{t}^{(i+1)}, K_{\infty}^{(i+1)}) \leq \epsilon_{i+1}.
\ee 

Then for any starting point $X_{T_{b}^{(i)}}$ of the chain at level $i$, we have
\be 
\P_{i}\Big[\sup_{T_{b}^{(i)} + S_{i+1} \leq t \leq T_{b}^{(i)} + S_{i+1} + S_{i}} D(K_{\infty}^{(i)}, K_{t}^{(i)}) < \epsilon_{i} \vert \mathcal{A} \Big] \geq 1 - \delta_{i}.
\ee  

\end{cor}
\begin{proof}

By inequality \eqref{IneqIndHypFinalSimpleBound} and Lemma \ref{LemmIndErr1}, we have for  $S_{i+1} > \mathcal{H}_{3}(\epsilon_{i}, \delta_{i}, S_{i})$, $\epsilon_{i+1}^{-1} > \mathcal{H}_{4}(\epsilon_{i}, \delta_{i}, S_{i})$ and $\mathcal{B}$ as defined in equation \eqref{IneqIndHypFinalSimpleBound} that
\be 
\P_{i}\Big[\sup_{T_{b}^{(i)} + S_{i+1} \leq t \leq T_{b}^{(i)} + S_{i+1} + S_{i}} & D(K_{\infty}^{(i)}, K_{t}^{(i)}) > \epsilon_{i} \vert \mathcal{A} \Big] \\
&\leq \P_{i} \Big[\sup_{T_{b}^{(i)} + S_{i+1} \leq t \leq T_{b}^{(i)} + S_{i+1} + S_{i}} W_{d}(F_{t}^{(i)}, \pi_{i}) > \frac{\epsilon_{i}^{2}}{\mathcal{B}^{2} p_{ee}^{2}} \vert \mathcal{A} \Big] \\
&\leq \delta_{i},
\ee 
finishing the proof.
\end{proof}

We are now ready to prove Theorem \ref{ThmNonQuantConv}. 

\begin{proof}[Proof of Theorem \ref{ThmNonQuantConv}]

We prove both parts of Theorem \ref{ThmNonQuantConv} by finding sequences $\{ G_{i}, T_{b}^{(i)} \}_{i=0}^{\mathcal{K}}$  with the property that  $K_{t}^{(i)}$ is close to $K_{\infty}^{(i)}$ for all times $T_{b}^{(i)} \leq t \leq T_{b}^{(i)} + G_{i}$ and all $0 \leq i \leq \mathcal{K}$ with high probability. \\ 
We first fix an approximation level $0 < \epsilon_{0} < 1$ and failure bound $0 < \delta < 1$.  We then call a sequence of constants $\{ G_{i}, B_{i} , T_{b}^{(i)}, \epsilon_{i} \}_{i=0}^{\mathcal{K}}$ a \textit{good sequence with boundary $\epsilon_{0}, \delta$} if $G_{\mathcal{K}} = \infty$, $B_{\mathcal{K}} = T_{b}^{(\mathcal{K})} = 0$, and the remaining terms satisfy the inequalities
\be \label{IneqGoodSeqDef}
\epsilon_{i}^{-1} &\geq \mathcal{H}_{3}(\epsilon_{i-1}, \frac{\delta}{\mathcal{K} + 1}, G_{i}) \\
B_{i} &\geq \mathcal{H}_{4}(\epsilon_{i-1}, \frac{\delta}{\mathcal{K} + 1}, G_{i}) \\
T_{b}^{(i-1)} &\geq T_{b}^{(i)} + B_{i} \\
T_{b}^{(i-1)} &\leq T_{b}^{(i)} + (B_{i} - B_{i-1}) + (G_{i} - G_{i-1}),
\ee 
where the functions $\mathcal{H}_3$ and $\mathcal{H}_4$ are as defined in Equation \eqref{eqn:consh3h4}.
We recall the definition of the event $\mathcal{A}(\epsilon_{i+1}, T_{b}^{(i)}, S_{i+1}, S_{i})$ given given immediately before inequality \eqref{IneqInd1Hyp1} in the statement of Lemma \ref{LemmIndErr1} and define the events $\mathcal{A}^{(i)} = \mathcal{A}(\epsilon_{i+1}, T_{b}^{(i)}, B_{i}, G_{i})$.
\begin{prop} \label{PropMainThmMainCalc}
Assume that $\{ G_{i}, B_{i} , T_{b}^{(i)}, \epsilon_{i} \}_{i=0}^{\mathcal{K}-1}$ are a \textit{good sequence} for some value $\epsilon_{0}, \delta$ (\textit{i.e.}, they satisfy inequalities \eqref{IneqGoodSeqDef}). For all $0 \leq i \leq \mathcal{K}$, 
\be \label{IneqPfThm7MainProp}
\P\Big[\sup_{T_{b}^{(i)} + B_{i} \leq t \leq T_{b}^{(i)} + B_{i} + G_{i}} D(K_{\infty}^{(i)}, K_{t}^{(i)}) < \epsilon_{i}  \Big] &\geq 1 - \frac{\delta (\mathcal{K} - i+1)}{\mathcal{K} + 1}.
\ee 
\end{prop}
\begin{proof}
We prove inequality \eqref{IneqPfThm7MainProp} by induction on $i$. For $i = \mathcal{K}$, the inequality is trivial, as $K_{t}^{(\mathcal{K})} = K_{\infty}^{(\mathcal{K})}$ for all $t \in \mathbb{N}$. Fix $0 \leq j < \mathcal{K}$ and assume that inequality \eqref{IneqPfThm7MainProp} holds for all $i > j$; we will show that it holds for $i=j$ as well. By Corollary \ref{LemmIndErrFinal},
\be 
\P_{j}\big[ \mathcal{A}^{(j)} \vert \cap_{i \geq j+1} \mathcal{A}^{(i)} \big] \geq 1 - \frac{ \delta}{\mathcal{K} +1}.
\ee 
By our induction hypothesis, this means that:
\be 
\P \big[ \cap_{i \geq j} \mathcal{A}^{(j)} \big] &\geq 1 - \frac{\delta ( \mathcal{K} - j)}{\mathcal{K} + 1} - \left(1 - \P_{j}\big[ \mathcal{A}^{(j)} \vert \cap_{i \geq j+1} \mathcal{A}^{(j)} \big] \right)\\
&\geq 1 - \frac{\delta ( \mathcal{K} - j)}{\mathcal{K} + 1}  - \frac{\delta}{\mathcal{K} +1}\\
&\geq 1 - \frac{\delta ( \mathcal{K} - j + 1)}{\mathcal{K} + 1},
\ee 
which completes the proof of the proposition.
\end{proof}
We now complete the proof of the two inequalities.  Noting that $\epsilon_{i} \leq \epsilon_{0}$, we have by Proposition \ref{PropMainThmMainCalc} that for all $0 \leq i \leq \mathcal{K}$ and all good sequences with boundary $\epsilon_{0}, \delta$,
\be \label{IneqMainConclusionMainTheorem}
\P\big[\sup_{T_{b}^{(i)} + B_{i} \leq t \leq T_{b}^{(i)} + B_{i} + G_{i}} D(K_{\infty}^{(i)}, K_{t}^{(i)}) < \epsilon_{0}  ] &\geq 1 - \delta. \\
\ee 
We now prove the existence of \textit{good sequences} with desirable properties. Fix any integer $G > 1$ and constant $\beta > 0$ and define $\epsilon_{0} = \delta = \beta$ and $G_{0} = G, B_{0} = 0$. We then define a sequence inductively for $i \geq 1$ by iteratively assigning:

\be \label{IneqGoodSeqDefAux}
G_{i} &= 2iG_{i-1} + 2iB_{i-1} + 2i \\
\epsilon_{i}^{-1} &= \Big\lceil \mathcal{H}_{3}(\epsilon_{i-1}, \frac{\beta}{\mathcal{K} + 1}, G_{i}) \Big\rceil\\
B_{i} &= \Big\lceil \mathcal{H}_{4}(\epsilon_{i-1}, \frac{\beta}{\mathcal{K} + 1}, G_{i}) \Big\rceil.\\
\ee 
Since each term on the right-hand side depends only on terms that are of lower index in $i$, or defined higher on the list of equations, or both, this does indeed define a sequence. Furthermore, such a sequence automatically satisfies the first two inequalities in \eqref{IneqGoodSeqDef} and guarantees that there is space in between the upper and lower bounds given in the third and fourth inequalities. Thus, a sequence $T_{b}^{(i)}$ can be chosen to complete these terms to a good sequence.  For fixed $\eta$ and $T_{b}^{(0)}, B_{0}$ defined according to this sequence, set 
\be 
\beta(\eta) = \inf \{ \beta \, : \, T_{b}^{(0)} + B_{0} < \eta \}.
\ee 

We then write $\{ G_{i}(\eta), B_{i}(\eta) , T_{b}^{(i)}(\eta), \epsilon_{i}(\eta) \}_{i=0}^{\mathcal{K}-1}$ for the sequence defined in this way with boundary $\epsilon_{0} = \delta = 2 \beta(\eta)$; by definition, this sequence satisfies $T_{b}^{(0)}(\eta) + B_{0}(\eta) < \eta$. Also define 
\be \label{EqDefTForRemark}
t(\eta) = T_{b}^{(0)}(\eta) + B_{0}(\eta) + 1 \leq \eta.
\ee

The above construction implies that $\lim_{\eta \rightarrow \infty} \beta(\eta) \leq \beta$. Since this holds for all $\beta > 0$, we have that 
\be \label{EqLimEtaGammaThm7}
\lim_{\eta \rightarrow \infty} \beta(\eta) = 0.
\ee

By inequality \eqref{IneqMainConclusionMainTheorem}, then
\be 
\lim_{\eta \rightarrow \infty} \P[D(K_{t(\eta)}^{(0)}, K_{\infty}^{(0)}) > 2 \beta(\eta)] \leq \lim_{\eta \rightarrow \infty} 2 \beta(\eta) = 0.
\ee
Thus, the sequence $T_{b}^{(i)}(2 \beta(\eta))$, $t(2 \beta(\eta))$ satisfies inequality \eqref{ThmNonQuantConvPt2}, the first part of  Theorem \ref{ThmNonQuantConv}. 

\begin{remark}
If we want to choose $t(\eta) \geq \eta - 2$, we note that we can modify a good sequence by adding any constant $C$ to $B_{\mathcal{K}}$ and the same constant to $\{ T_{b}^{(i)} \}_{0 \leq i < \mathcal{K}}$ and $t$; all of the resulting bounds still hold as stated. 
\end{remark}

To prove inequality \eqref{ThmNonQuantConvPt2}, the second part of  Theorem \ref{ThmNonQuantConv}, we follow the same arguments used in the proof of Lemma \ref{LemmIndErr1} with $i = 0$. Fixing a good sequence and briefly following the argument for an arbitrary good sequence, we apply the bound \eqref{IneqMainConclusionMainTheorem} with $i = 1$ to Theorem \ref{ThmConcAlmostPositiveAlt}, exactly as it was used to obtain inequality \eqref{IneqInd1ConcBound1}, and then applying Corollary \ref{CorContPow} (again, exactly as used to obtain inequality \eqref{IneqInd1ConcBound1Alt} from inequality \eqref{IneqInd1ConcBound1}). Conditional on $\mathcal{A}^{(0)}$, this gives us the bound:
\be 
\P_{0}\Big[\Big \vert \frac{1}{G_{0}+1} \sum_{u=0}^{ S_{1} } f(X_{T_{b}^{(0)} + u}^{(0)}) - \pi_{0}(f) \Big \vert \geq k r +  \mathcal{D} \vert \mathcal{A}^{(0)} \Big] \leq 2k e^{-\alpha r^{2} \lfloor \frac{ G_{0} + 1 }{k} \rfloor}.\\
\ee
Choose $G_{0} = s-1$, and  $\epsilon_{0}(\eta) = \delta(\eta) \equiv 2 \beta(\eta)$ as above, this gives the unconditional bound  
\be \label{BlahAlmostDoneDraftIneq}
\P\Big[\Big \vert \frac{1}{G_{0}+1} \sum_{u=0}^{ S_{1} } f(X_{T_{b}^{(0)} + u}^{(0)}) - \pi_{0}(f) \Big \vert \geq k r +  \mathcal{D} \vert \Big] \leq  2k e^{-\alpha r^{2} \lfloor \frac{ G_{0} + 1 }{k} \rfloor}.\\
\ee
We now consider the good sequence $\{ G_{i}(\eta), B_{i}(\eta) , T_{b}^{(i)}(\eta), \epsilon_{i}(\eta) \}_{i=0}^{\mathcal{K}-1}$ defined above. From the definition of $\mathcal{H}_{3}, \mathcal{H}_{4}$ it is clear that 
\be 
\lim_{\eta \rightarrow \infty} S_{i}(\eta) &= \infty \\
\lim_{\eta \rightarrow \infty} \epsilon_{i}(\eta) &\leq \lim_{\eta \rightarrow \infty} \epsilon_{0}(\eta) = 0,
\ee 
and so the associated constant $\mathcal{D} = \mathcal{D}(\epsilon_{i}(\eta), S_{i}(\eta)) \equiv \mathcal{D}(\eta)$ defined in equation \eqref{EqDefScriptD} satisfies
\be 
\lim_{\eta \rightarrow \infty} \mathcal{D}(\eta) = 0.
\ee  
Noting that the limit in Equation \eqref{EqLimEtaGammaThm7} holds and applying this bound to inequality \eqref{BlahAlmostDoneDraftIneq} yields,
\be 
\lim_{\eta \rightarrow \infty} \P\Big[\Big \vert \frac{1}{s+1} \sum_{u=0}^{ s} f(X_{T_{b}(\eta)^{(0)} + u}^{(0)}) - \pi_{0}(f) \Big \vert \geq k r  \Big] \leq  2k e^{-\alpha r^{2} \lfloor \frac{ s }{k} \rfloor}.
\ee
This completes the proof.
\end{proof}

\begin{remark}
We note that the proof of Theorem \ref{ThmNonQuantConv} involves rather poor bounds on the required burn-in time for convergence. A large part of the problem is that the method of proof requires $K_{t}^{(i)}$ to have converged to $K_{\infty}^{(i)}$ in the strict metric $D$ given in Equation \eqref{EqDistOnKernDef}. In many examples, including the saw-tooth potential described in Figure \ref{fig:sawtooth}, the equi-energy sampler mixes long before $D(K_{t}^{(0)}, K_{\infty}^{(0)})$ is small. One partial solution is to use a weaker metric; see the second problem in Section 7.\end{remark}

Many natural limiting chains have the property that they have negative curvature while a small power of their transition kernels have strictly positive curvature. We give here an archetypal example below, together with a calculation that can be used to prove similar curvature bounds for many other multimodal examples.
\begin{example} Let $\Omega$ be the unit circle with metric $d$ given by equation \eqref{DefDistOnTorus}, fix a constant $\mathcal{C} > 0$, and consider the `saw-tooth' potential:
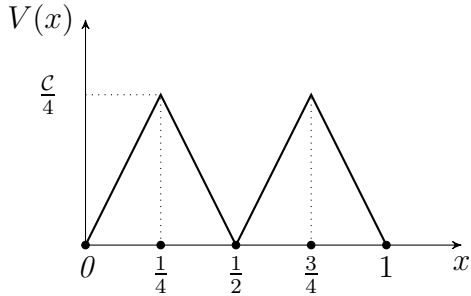
\begin{figure}[h] 
\begin{tikzpicture}
\draw[->] (0,0) -- (5,0) node[anchor=north] {$x$};
\draw	(0,0) node[anchor=north] {0}
		(1,0) node[anchor=north] {$\frac{1}{4}$}
		(2,0) node[anchor=north] {$\frac{1}{2}$}
		(3,0) node[anchor=north] {$\frac{3}{4}$}
		(4,0) node[anchor=north] {$1$};
		
\draw        (-0.5,2) node{$\frac{\mathcal{C}}{4}$};

\draw[->] (0,0) -- (0,3) node[anchor=east] {$V(x)$};
\draw [fill] (0,0) circle [radius=.05];
\draw [fill] (1,0) circle [radius=.05];
\draw [fill] (2,0) circle [radius=.05];
\draw [fill] (3,0) circle [radius=.05];
\draw [fill] (4,0) circle [radius=.05];

\draw[thick] (0,0) -- (1,2) -- (2,0) -- (3,2) -- (4,0);
\draw[dotted] (3,0) -- (3,2);
\draw[dotted] (1,0) -- (1,2);
\draw[dotted] (0,2) -- (1,2);
\end{tikzpicture}
\caption{`Saw-tooth' potential}
\label{fig:sawtooth}
\end{figure}
\be[eqn:potV]
V(x) &= \mathcal{C}x, \quad x \in \left[0, \frac{1}{4} \right]; \\
    V(x)     &= \frac{\mathcal{C}}{4} - \mathcal{C} \left( x - \frac{1}{4} \right), \quad x \in \left[ \frac{1}{4}, \frac{1}{2} \right];\\
         V(x) &= V \left( x - \frac{1}{2} \right), x \in  \left[ \frac{1}{2},1 \right].
 \ee
Define $K_{MH}$ as in equation \eqref{eqn:Kmet}. We assume $c^{-1}, \mathcal{C}$ are both very large, and define energy rings by the intervals $H(v) = \left( v-\frac{c}{4}, v+ \frac{c}{4} \right)$. Let $\{ X_{t} \}_{t \in \mathbb{N}}$ be a Markov chain run by the limiting kernel $K_{\infty}^{(0)}$ associated with these choices of $V(x)$, $K_{\MH}$ and $H$. Define $\mathcal{P}(x) = x - \frac{1}{2} \textbf{1}_{x > \frac{1}{2}}$. Then the projection $\{ \widehat{X}_{t} \}_{t \in \mathbb{N}} = \{ \mathcal{P}(X_{t}) \}_{t \in \mathbb{N}}$ of $ \{X_{t} \}_{t \in \mathbb{N}}$ to the interval $[0, \frac{1}{2}]$ is also a Markov chain. Furthermore, the chain $\{ X_{t} \}_{t \in \mathbb{N}}$ `forgets' all information not contained in $\{ \widehat{X}_{t} \}_{t \in \mathbb{N}}$ as soon as an equi-energy step is taken, in the following sense:  conditioned on an equi-energy move at time $t$, we have for all measurable sets $A$
  \be
  \P[X_{t+1} \in \mathcal{P}^{-1}(A) \cap [0, \frac{1}{2}] ] = \frac{1}{2} \P[\widehat{X}_{t+1} \in A].
  \ee
    It is easy to check that $- \kappa_{\MH} = \beta > 0$, $\kappa_{\EE} \geq 0$, and that $\mathcal{P} K_{\infty}^{(0)}$ has curvature $\alpha > 0$. Couple two chains $X_{t}, Y_{t}$ driven by $K_{\infty}^{(0)}$ so that they always take equi-energy moves at the same time, their equi-energy moves are coupled according to the quantile coupling, and their MH moves are coupled by the optimal 1-step coupling under the projection $\mathcal{P}$. We then note that, describing contraction by the first step $q$ at which an equi-energy move occurs, we have
\be 
\E[d(X_{k}, Y_{k}) &\vert X_{0} = x, Y_{0} = y] \\
&\leq \sum_{q=1}^{k} (1 + \beta)^{q-1} (1 - \alpha)^{k-q} p_{ee} (1-p_{ee}^{q-1}) d(x,y).
\ee 
Thus, for $k$ sufficiently large, the curvature $\kappa_{\infty}^{0,k}$ is strictly positive. This calculation is not specific to saw-tooth potentials; the same argument applies to any repeating finite potential wells for which the projected Metropolis-Hastings chain has positive curvature.\qed
\end{example}

We take the opportunity to use the target density given by \eqref{eqn:targleb} in Section \ref{SubsecFinVsInfAut} as a basic check for the asymptotic efficiency of the method used to prove Theorem \ref{ThmNonQuantConv}. In particular, we would like to know how quickly $D(K_{t}^{(0)}, K_{\infty}^{(0)})$ decays as a function of $t$ when $t$ is very large, rather than the smallest $t$ for which that quantity is small, which is the focus of this paper. We note that, after running $X_{t}^{(1)}$ for $T$ steps, standard concentration results yield that, for some $\mathcal{C} > 0$, all $ 0 < r < 1$ and all 1-Lipschitz functions $f$, 
\be 
\P[\vert F_{T}^{(1)}(f) - \pi(f) \vert > r] \leq 2 e^{-\mathcal{C}r^{2}T}.
\ee 
 Here $F_T^{(1)}$ denotes the empirical measure of 
$\{X^{(1)}_t\}_{t \leq T}$.
Fix $\delta > 0$. By the same covering argument as described around equation \eqref{IneqUsingKantRubAbs}, there exist $N(\delta)$ with $\log (N(\delta)) = O(1/\delta)$ and functions $\{ f_{i} \}_{i=1}^{N(\delta)}$ such that, for all measures $\mu, \nu$ and all 1-Lipschitz functions $f$, we have
\be 
\vert \mu(f) - \nu(f) \vert \leq \delta + \sup_{1 \leq i \leq N(\delta)} \vert \mu(f_{i}) - \nu(f_{i}) \vert. 
\ee 
Combining these two bounds, we use the Kantorovitch-Rubinstein duality theorem to find that for $T > \delta^{-2} \log \left( \frac{e^{\delta^{-1}}}{\delta} \right) \approx \delta^{-3}$, we have
\be 
W_{d}(F_{T}^{(1)}, \pi) &\leq \sup_{\| f \|_{\lip} \leq 1} \vert F^{(1)}_{T}(f) - \pi(f) \vert \\
&\leq \delta + \sup_{1 \leq i \leq N(\delta)} \vert F^{(1)}_{T}(f_{i}) - \pi(f_{i}) \vert \\
&\leq 3 \delta. 
\ee 
Thus, in this simple example our method gives an error of $O \left(T^{-{1 \over 3}} \right)$ for the distance between the empirical distribution of the top chain and its corresponding invariant measure. Applying Lemma \ref{LemWassAppKern} as written, this implies that $D(K_{T}^{(0)}, K_{\infty}^{(0)}) = O \left( T^{-\frac{1}{6}} \right)$ (a more careful application of the same argument gives  $D(K_{T}^{(0)}, K_{\infty}^{(0)}) = O \left( T^{-\frac{1}{3}} \right)$). As shown in section \ref{SubsecFinVsInfAut}, the convergence rate of the empirical estimate is $O \left( T^{- \frac{1}{2}} \right)$. Thus, our technique loses something even for this very simple example, but is certainly tight enough to distinguish between `rapid' and `slow' mixing.

\section{Discussion} \label{SecDisc}
This paper is a first effort towards using coupling techniques to find useful quantitative bounds on the mixing properties of adaptive algorithms. There are immediate open problems that we believe are accessible from our work. We list a few of them below.
\begin{enumerate} \label{ListOfProblems}
\item The notion of curvature in this paper can be used to analyze other adaptive algorithms. The tools in this paper do not apply directly to algorithms without the strong conditional independence properties of the equi-energy sampler. However, for many other algorithms, such as the Wang-Landau algorithm, it may be possible to find related conditional near-independence conditions over small time intervals. 

\item In our examples, we considered only the case when the state space $\Omega$ is compact;
for non-compact $\Omega$, many further challenges remain. The main difficulty is that the most obvious analogue to Lemma \ref{LemWassAppKern} is false. Fortunately, it is enough for Lemma \ref{LemWassAppKern} to hold on compact sets; this combined with drift conditions should give concentration bounds. 
\item  We point out that it is possible to simulate rigorous error bounds for the equi-energy sampler by following the strategy set out in \cite{CoRo98}. It would be interesting to find situations under which we expect that method, or refinements, to work well. A particular area of interest for both simulated and theoretical bounds is finding ways in which nice properties of the limiting chain, such as monotonicity, might be transferred to the equi-energy sampler. 
\item As shown in subsection \ref{SubsecFinVsInfAut}, an equi-energy sampler can sometimes give samples that are worse than that of its underlying Markov chain. Are there situations under which we can guarantee that the equi-energy sampler is not `too much' worse over any time scale?
\end{enumerate}

\section{Appendix}
\begin{proof}[Proof of Lemma \ref{CorCoupPastCon}]
Fix $\gamma > 0$. We construct a coupling of $X \sim \mu \vert_{B}$ and $Y \sim \nu \vert_{B}$ as follows. Let $X', Y'$ be a coupling of $\mu, \nu$ so that $\E[d(X',Y')] \leq \gamma + W_{d}(\mu, \nu)$. Next, condition on the event $X' \in B$, and set $x = X'$. Then let $Z_{1}$ be Bernoulli random variable, independent of $X'$, with success probability $\P[Y' \in B_{\delta} \vert X'=x]$. If $Z_{1} = 1$, choose $Y'$ from the distribution $R(\cdot) \equiv \P[Y' \in \cdot \vert X'=x, Y' \in B_{\delta}]$. Otherwise, choose $Y'$ from its remainder distribution. Finally, let $Z_{2}$ be the indicator function for $\{ Y' \in B \}$. If $Z_{2} =1$, choose $Y = Y'$. Otherwise, choose $Y$ independently of $X'$ from its remainder distribution. Note that, for all $A \subset \Omega$,
\be 
\P[Y \in A, Z_{1} = 1, Z_{2} = 1] &= \int_{B} \P[Y' \in A \vert X' = x] dx \\
&\leq \P[Y' \in A],
\ee 
and so such a remainder distribution for $Y$ exists. We then have
\be 
\E[d(X,Y)] &= \E[d(X,Y) \vert Z_{1} = 0]\P[Z_{1} = 0]  \\
&\hspace{1cm}+ \E[d(X,Y) \vert Z_{1} = 1, Z_{2} = 0]\P[Z_{1} = 1, Z_{2} = 0] \\
&\hspace{1cm}+ \E[d(X,Y) \vert Z_{1} = 1, Z_{2} = 1]\P[Z_{1} = 1, Z_{2} = 1] \\
&\leq \mathrm{diam}(\Omega) \frac{\epsilon + \gamma}{\delta} +  \E[d(X,Y) \vert Z_{1} = 1, Z_{2} = 0]\P[Z_{1} = 1, Z_{2} = 0] \\
&\hspace{1cm}+ \E[d(X,Y) \vert Z_{1} = 1, Z_{2} = 1]\P[Z_{1} = 1, Z_{2} = 1] \\
&\leq \mathrm{diam}(\Omega) \frac{\epsilon + \gamma}{\delta} + \mathrm{diam}(\Omega) \nu(B_{\delta} \backslash B) + \frac{\epsilon \gamma }{\mu(B)},
\ee 
where the first inequality is due to Lemma \ref{LemmaCoupPastBound}. Letting $\gamma$ go to 0 completes thr proof.
\end{proof}

\section*{Acknowledgements}
The authors thank Luke Bornn, Steve Finch, Martin Hairer, Jonathan Mattingly, Gareth Roberts, Andrew Stuart and Dawn Woodard for useful conversations. Part of this work was done when NSP was visiting the Division of Applied Mathematics at
Brown University and AS was visiting the Department of Statistics at Harvard University. We thank these institutions for their hospitality. NSP thanks Prof. Kavita Ramanan for her kind invitation to Brown University. NSP is partially supported by the NSF grant DMS-1107070.
\bibliographystyle{plain}
\bibliography{CurvBIB}
\end{document}